%% file: dp_directed_jgt.tex
\documentclass[11pt]{article}
\usepackage[a4paper,left=3cm, right=3cm,top=3.5cm, bottom=4.5cm]{geometry}
\usepackage{amsmath,amssymb,amsfonts, marvosym}
\usepackage{epsfig}
\usepackage{tikz}
\usepackage{graphics}
\usepackage{bbm}
\usepackage{bm, bbm}
\usepackage{algorithm}
\usepackage{algpseudocode}
\usepackage{todonotes}
\usepackage{complexity}
\usepackage[thmmarks]{ntheorem}
\usepackage[]{hyperref}
 \hypersetup{colorlinks=true, citecolor=black,%
                 filecolor=black,%
                 linkcolor=black,%
                 urlcolor=black,  pdftex}

\usetikzlibrary[graphs, arrows, backgrounds, intersections, positioning, fit, petri, calc, shapes, decorations.pathmorphing]
\usetikzlibrary{arrows.meta}
\usepgflibrary{patterns}
\tikzset{node distance=1cm, bend angle=20,
vertex/.style={circle,minimum size=2mm,very thick, draw=black, fill=black, inner sep=0mm}, information text/.style={inner sep=1ex, font=\Large}, help lines/.style={-,color=black, >=stealth', shorten <=.5pt, shorten >=.5pt}, blue help lines/.style={help lines,color=darkblue}, red help lines/.style={help lines, color=darkred}}

\definecolor{darkred}{RGB}{105,0,0}
\usepackage[retainorgcmds]{IEEEtrantools}

\makeatletter
\newtheoremstyle{prime}%
 {\item[\hskip\labelsep \theorem@headerfont ##1\ \theorem@separator]}%
{\item[\hskip\labelsep \theorem@headerfont ##1\ ##3' \theorem@separator]}
\makeatother

\makeatletter
\newtheoremstyle{proofof}
{\item[\hskip\labelsep \theorem@headerfont ##1\ \theorem@separator]}%
{\item[\hskip\labelsep \theorem@headerfont ##1\ ##3\theorem@separator]}
\makeatother

\newtheorem{theorem}{Theorem}

\newtheorem{proposition}[theorem]{Proposition}

\newtheorem{corollary}[theorem]{Corollary}

\newtheorem{claim}{Claim}

\theoremstyle{prime}

\def \QD1 {\hfill $\spadesuit$}

\newcommand{\case}[2]{\noindent {\bf Case #1\/:} {\it #2}}

\numberwithin{equation}{section}

\theoremstyle {nonumberplain}
\theoremseparator{:}
\theorembodyfont{\normalfont}
\theoremsymbol{\rule{1ex}{1ex}}
\newtheorem{proof}{Proof}

\theoremstyle{proofof}

\theoremsymbol{$\square$}
\newtheorem{proof2}{Proof}

\theoremsymbol{$\diamond$}

\begin{document}
\title{\bf On DP-Coloring of Digraphs}

\author{{{J{\o}rgen Bang-Jensen}\thanks{Research supported by the Danish
research council under grant number 7014-00037B}
\thanks{
University of Southern Denmark, IMADA, Campusvej 55, DK-5320 Odense M, Denmark. E-mail address: jbj@imada.sdu.dk
}}
\and
{{Thomas Bellitto}\footnotemark[1]
\thanks{
University of Southern Denmark, IMADA, Campusvej 55, DK-5320 Odense M, Denmark. E-mail address: bellitto@imada.sdu.dk
}}
\and
{{Thomas Schweser
}\thanks{
Technische Universit\"at Ilmenau, Inst. of Math., PF 100565, D-98684 Ilmenau, Germany. E-mail
address: thomas.schweser@tu-ilmenau.de}}
\and
{{Michael Stiebitz}\thanks{
Technische Universit\"at Ilmenau, Inst. of Math., PF 100565, D-98684 Ilmenau, Germany. E-mail
address: michael.stiebitz@tu-ilmenau.de}}
}

\date{}
\maketitle

\begin{abstract}
DP-coloring is a relatively new coloring concept by Dvo\v{r}\'ak and Postle and was introduced as an extension of list-colorings of (undirected) graphs.  It transforms the problem of finding a list-coloring of a given graph $G$ with a list-assignment $L$ to finding an independent transversal in an auxiliary graph with vertex set $\{(v,c) ~|~ v \in V(G), c \in L(v)\}$. In this paper, we extend the definition of DP-colorings to digraphs using the approach from Neumann-Lara where a coloring of a digraph is a coloring of the vertices such that the digraph does not contain any monochromatic directed cycle. Furthermore, we prove a Brooks' type theorem regarding the DP-chromatic number, which extends various results on the (list-)chromatic number of digraphs. 
\end{abstract}



\section{Introduction}
Recall that the chromatic number $\chi(G)$ of an undirected graph $G$ is the least integer $k$ for which there is a coloring of the vertices of $G$ with $k$ colors such that each color class induces an edgeless subgraph of $G$. The chromatic number $\chi(D)$ of a digraph $D$, as defined in \cite{NeuLa82} by Neumann-Lara, is the smallest integer $k$ for which there is a coloring of the vertices of $D$ with $k$ colors such that each color class induces an acyclic subdigraph of $D$, \textit{i.e.}, a subdigraph that does not contain any directed cycle. This definition is especially reasonable because it implies that the chromatic number of a bidirected graph and the chromatic number of its underlying (undirected) graph coincide. Furthermore, it shows that various results concerning the chromatic number of undirected graphs can be extended to digraphs. For example, the analogue to Brooks' famous theorem \cite{brooks} that the chromatic number of a graph is always at most its maximum degree plus 1 and that the only conncected graphs for which equality hold are the complete graphs and the odd cycles was proven by Mohar \cite{Mo10}. As usual, a digraph $D$ is $k$\textbf{-critical} if $\chi(D)=k$ but $\chi(D') \leq k-1$ for every proper subdigraph $D'$ of $D$. Mohar~\cite{Mo10} proved the following:

\begin{theorem}[Mohar 2010]
Suppose that $D$ is a $k$-critical digraph in which each vertex $v$ satisfies $d_D^+(v)=d_D^-(v)=k-1$. Then, one of the following cases occurs:
\begin{itemize}
\item[\upshape (a)] $k=2$ and $D$ is a directed cycle of length $\geq 2$.
\item[\upshape (b)] $k=3$ and $D$ is a bidirected cycle of odd length $\geq 3$.
\item[\upshape (c)] $D$ is a bidirected complete graph.
\end{itemize}
\end{theorem} 

Moreover, some results regarding the list-chromatic number can also be transferred to digraphs. Given a digraph $D$, some color set $C$, and a function $L: V(D) \to 2^C$ (a so-called \textbf{list-assignment}), an $L$\textbf{-coloring} of $D$ is a function $\varphi:V(D) \to C$ such that $\varphi(v) \in L(v)$ for all $v \in V(D)$ and  $D[\varphi^{-1}(\{c\})]$ contains no directed cycle for each $c \in C$ (if such a coloring exists, we say that $D$ is  $L$\textbf{-colorable}). Harutyunyan and Mohar \cite{HaMo11} proved the following, thereby extending a theorem of Erd\H{o}s, Rubin and Taylor \cite{ErdRubTay79} for undirected graphs. Recall that a \textbf{block} of a digraph is a maximal connected subdigraph that does not contain a separating vertex.

\begin{theorem}\label{theorem_harut}
Let $D$ be a connected digraph, and let $L$ be a list-assignment such that $|L(v)| \geq \max \{d_D^+(v), d_D^-(v)\}$ for all $v \in V(D)$. Suppose that $D$ is not $L$-colorable. Then, $D$ is Eulerian and for every block $B$ of $D$ one of the following cases occurs:
\begin{itemize}
\item[\upshape (a)] $B$ is a directed cycle of length $\geq 2$.
\item[\upshape (b)] $B$ is a bidirected cycle of odd length $\geq 3$.
\item[\upshape (c)] $B$ is a bidirected complete graph.
\end{itemize}
\end{theorem}

Recently, Dvo\v{r}\'ak and Postle \cite{DvoPo15} introduced a new coloring concept, the so-called DP-colorings (they call it correspondence colorings). DP-colorings are an extension of list-colorings, which is based on the fact that the problem of finding an $L$-coloring of a graph $G$ can be transformed to that of finding an appropriate independent set in an auxiliary graph with vertex set $\{(v,c)~|~v \in V(G), c \in L(v)\}$. In Section \ref{sec-dp_intro}, we extend  the concept of DP-coloring from graphs to digraphs. In particular, we introduce the DP-chromatic number of a digraph and show that the DP-chromatic number of a bidirected graph is equal to the DP-chromatic number of its underlying graph (see Corollary~\ref{cor_dp-coincides}). As the main result of our paper we provide a characterization of DP-degree colorable digraphs (see Theorem~\ref{theorem:main-result} and Theorem~\ref{theorem_harut-extended}) that generalizes Theorem~\ref{theorem_harut}.

\section{Basic Terminology}
 For an extensive depiction of digraph terminology we refer the reader to \cite{BaGu08}. Given a digraph $D$, we denote the  \textbf{set of vertices} of $D$ by $V(D)$ and the \textbf{set of arcs} of $D$ by $A(D)$ . The number of vertices of $D$ is called the \textbf{order} of $G$ and ist denoted by $|D|$. Digraphs in this paper may not have loops nor parallel arcs; however, it is allowed that there are two arcs going in opposite directions between two vertices (in this case we say that the arcs are \textbf{opposite}). We denote by $uv$ the arc whose \textbf{initial vertex} is $u$ and whose \textbf{terminal vertex} is $v$; $u$ and $v$ are also said to be the \textbf{end-vertices} of the arc $uv$. Let $X,Y\subseteq V(D)$, then $E_D(X,Y)$ denotes the set of arcs that have their initial vertex in $X$ and their terminal vertex in $Y$. Two vertices $u,v$ are \textbf{adjacent} if at least one of $uv$ and $vu$ belongs to $A(D)$. If $u$ and $v$ are adjacent, we also say that $u$ is a \textbf{neighbor} of $v$ and vice versa. If $uv \in A(D)$, then we say that $v$ is an \textbf{out-neighbor} of $u$ and $u$ is an \textbf{in-neighbor} of $v$. By $N_D^+(v)$ we denote the set of out-neighbors of $v$; by $N_D^-(v)$ the set of in-neighbors of $v$. Given a digraph $D$ and a vertex set $X$, by $D[X]$ we denote the subdigraph of $D$ that is \textbf{induced} by the vertex set $X$, that is, $V(D[X])=X$ and $A(D[X])=\{uv \in A(D) ~|~ u,v \in X\}$. A digraph $D'$ is said to be an induced subdigraph of $D$ if $D'=D[V(D')]$. As usual, if $X$ is a subset of $V(D)$, we define $D-X=D[V(D) \setminus X]$. If $X=\{v\}$ is a singleton, we use $D-v$ rather than $D- \{v\}$. The \textbf{out-degree} of a vertex $v \in V(D)$ is the number of arcs whose inital vertex is $v$; we denote it by $d_D^+(v)$. Similarly, the number of arcs whose terminal vertex is $v$ is called the \textbf{in-degree} of $v$ and is denoted by $d_D^-(v)$. Note that $d_D^+(v)=|N_D^+(v)|$ and $d_D^-(v)=|N_D^-(v)|$ for all $v \in V(D)$. A vertex $v \in V(D)$ is \textbf{Eulerian} if $d_D^+(v)=d_D^-(v)$. Moreover, the digraph $D$ is \textbf{Eulerian} if every vertex of $D$ is Eulerian. By $\Delta^+(D)$ (respectively $\Delta^-(D)$) we denote the \textbf{maximum out-degree} (respectively \textbf{maximum in-degree}) of $D$. A \textbf{matching} in $D$ is a set $M$ of arcs of $D$ with no common end-vertices. A matching in $D$ is \textbf{perfect} if it contains $\frac{|D|}{2}$ arcs.
 
Given a digraph $D$, its \textbf{underlying} graph $G(D)$ is the simple undirected graph with $V(G(D))=V(D)$ and $\{u,v\}\in E(G(D))$ if and only if at least one of $uv$ and $vu$ belongs to $A(D)$. The digraph $D$ is \textbf{(weakly) connected} if $G(D)$ is connected. A \textbf{separating vertex} of a connected digraph $D$ is a vertex $v \in V(D)$ such that $D-v$ is not connected. Furthermore, a \textbf{block} of $D$ is a maximal subdigraph $D'$ of $D$ such that $D'$ has no separating vertex. By $\mathcal{B}(D)$ we denote the set of all blocks of $D$. 
 
A \textbf{directed path} is a non-empty digraph $P$ with $V(P)=\{v_1,v_2,\ldots,v_p\}$ and $A(P)=\{v_1v_2, v_2v_3, \ldots, v_{p-1}v_p\}$ where the $v_i$ are all distinct. Furthermore, a \textbf{directed cycle} of  \textbf{length} $p\geq 2$ is a non-empty digraph $C$ with $V(C)=\{v_1,v_2,\ldots,v_p\}$ and $A(C)=\{v_1v_2,v_2v_3, \ldots, v_{p-1}v_p, v_pv_1\}$ where the $v_i$ are all distinct. A directed cycle of length $2$ is called a \textbf{digon}. If $D$ is a digraph and if $C$ is a cycle in the underlying graph $G(D)$, we denote by $D_C$ the maximal subdigraph of $D$ satisfying $G(D_C)=C$. A \textbf{bidirected} graph is a digraph that can be obtained from an undirected  (simple) graph $G$ by replacing each edge by two opposite arcs, we denote it by $D(G)$. A bidirected complete graph is also called a \textbf{complete digraph}.
 
\section{DP-Colorings of digraphs} \label{sec-dp_intro}
\subsection{The DP-Chromatic Number}
Let $D$ be a digraph. A \textbf{cover} of $D$ is a pair $(X,H)$ satisfying the following conditions:

\begin{itemize}
\item[\textbf{(C1)}] $X: V(D) \to 2^{V(H)}$ is a function that assigns to each vertex $v \in V(D)$ a vertex set $X_v = X(v) \subseteq V(H)$ such that the sets $X_v$ with $v \in V(D)$ are pairwise disjoint.

\item[\textbf{(C2)}] $H$ is a digraph with $V(H) = \bigcup_{v \in V(D)}X_v$ such that each $X_v$ is an independent set of $H$. For each arc $a=uv \in A(D)$, the arcs from $E_H(X_u,X_v)$ form a possibly empty matching $M_a$ in $H[X_u \cup X_v]$. Furthermore, the arcs of $H$ are $A(H)=\bigcup_{a \in A(D)} M_a$. 
\end{itemize}

Now let $(X,H)$ be a cover of $D$. A vertex set $T \subseteq V(H)$ is a \textbf{transversal} of $(X,H)$ if $|T \cap X_v|=1$ for each vertex $v \in V(D)$. An \textbf{acyclic transversal} of $(X,H)$ is a transversal $T$ of $(X,H)$ such that $H[T]$ contains no directed cycle. An acyclic transversal of $(X,H)$ is also called an $(X,H)$\textbf{-coloring} of $D$; the vertices of $H$ are called \textbf{colors}. We say that $D$ is $(X,H)$\textbf{-colorable} if $D$ admits an $(X,H)$-coloring. Let $f: V(D) \to \mathbb{N}_0$ be a function. Then, $D$ is said to be \textbf{DP-}$f$\textbf{-colorable} if $D$ is $(X,H)$-colorable for every cover $(X,H)$ of $D$ satisfying $|X_v| \geq f(v)$ for all $v \in V(D)$ (we will call such a cover an $f$\textbf{-cover}). If $D$ is DP-$f$-colorable for a function $f$ such that $f(v)=k$ for all $v \in V(D)$, then we say that $D$ is \textbf{DP-}$k$\textbf{-colorable}. The \textbf{DP-chromatic number} $\chi_{\text{DP}}(D)$ is the smallest integer $k \geq 0$ such that $D$ is DP-$k$-colorable. 

DP-coloring was originally introduced for undirected graphs by Dvor{\'a}k and Postle \cite{DvoPo15}. Let $G$ be an undirected (simple) graph. A \textbf{cover} of $G$ is a pair $(X,H)$ satisfying (C1) and (C2) where the matching $M_e$ associated to an edge $e=uv \in E(G)$ is an undirected matching between $X_u$ and $X_v$ (and $H$ is therefore an undirected graph). An $(X,H)$\textbf{-coloring} of $G$ is an \textbf{independent transversal} $T$ of $(X,H)$, \textit{i.e.}, $T$ is a transversal of $(X,H)$ such that $H[T]$ is edgeless. The definitions of DP-$f$-colorable, DP-$k$-colorable and the DP-chromatic number are analogous.

We now investigate the relation between undirected and directed DP-colorings.

\begin{theorem}\label{thm:undirected}
 A bidirected graph $D$ is DP-$f$-colorable if and only if its underlying undirected graph $G(D)$ is DP-$f$-colorable.
\end{theorem}

\begin{proof} We prove the two implications separately. First assume that $D$ is DP-$f$-colorable. In order to show that $G=G(D)$ is DP-$f$-colorable, let $(X,H_G)$ be an $f$-cover of $G$ and let  $H_D=D(H_G)$ be the bidirected  graph associated to $H_G$. Then, $(X,H_D)$ is an $f$-cover of $D$. By assumption, there is an acyclic transversal $T$ of $(X,H_D)$. As $H_D$ is bidirected, $T$ is an  independent transversal of $(X,H_G)$ and so $G$ is DP-$f$-colorable.
 
The converse is less obvious since even if $D$ is bidirected, its covers do not have to be bidirected. Let $(X,H_D)$ be a cover of a bidirected graph $D$. We say that the cover is \textbf{symmetric} if and only if for every pair of opposite arcs $uv$ and $vu$ in $D$, the matchings $M_{uv}$ and $M_{vu}$ are \textbf{opposite}, that is, each arc in $M_{vu}$ is opposite to some arc in $M_{uv}$. We say that the cover is \textbf{locally-symmetric} around a given vertex $v\in V(D)$ if $M_{uv}$ and $M_{vu}$ are opposite for every vertex $u$ adjacent to $v$. 

Let $f$ be such that $D$ is not DP-$f$-colorable. We claim that $G=G(D)$ is not DP-$f$-colorable. To prove this, we choose an $f$-cover $(X,H_D)$ of $D$ for which $D$ is not $(X,H_D)$-colorable such that $(X,H_D)$ is locally-symmetric around a maximum number of vertices. Suppose that there exists a vertex $v\in V(D)$ around which $(X,H_D)$ is not locally-symmetric. Let $(X,H'_D)$ be the $f$-cover of $D$ obtained from $(X,H_D)$ by replacing $M_{uv}$ by the opposite of $M_{vu}$ for every vertex $u$ adjacent to $v$ (note that this will not affect vertices that are already locally symmetric). By the the choice of $(X,H_D)$, there exists an acyclic transversal $T$ of $(X,H'_D)$. Then, $T$ is also a transversal of $(X,H_D)$, and, since $D$ is not $(X,H_D)$-colorable, $H_D[T]$ contains a directed cycle $C$.

As $H_D - X_v$ is isomorphic to $H_D' - X_v$, it follows from the choice of $T$ that $C$ must contain a vertex $x \in X_v$. Hence, there exists a vertex $u$ adjacent to $v$ in $D$ and a vertex $x'\in X_u$ such that $xx'\in M_{vu}$ and $x'\in T$. Since the graph $H'_D$ contains both the arcs $xx'$ and $x'x$, $H'_D[\{x,x'\}]$ is a digon and, hence, $H'_D[T]$ also contains a directed cycle. Thus, $(X,H'_D)$ is an $f$-cover of $D$ for which $D$ is not $(X,H'_D)$-colorable, but $(X,H'_D)$ is locally symmetric around strictly more vertices than $(X,H_D)$, contradicting the choice of $(X,H_D)$. Consequently, $(X,H_D)$ is symmetric and, as a consequence, for $H_G=G(H_D)$, the pair $(X,H_G)$ is an $f$-cover of the underlying graph $G=G(D)$ such that $G$ is not $(X,H_G)$-colorable, which implies that $G$ is not DP-$f$-colorable.
\end{proof}

An important property of the chromatic number of a digraph is that the chromatic number of a bidirected graph coincides with the chromatic number of its underlying graph. Theorem \ref{thm:undirected} implies that this property also holds for DP-coloring:

\begin{corollary}\label{cor_dp-coincides}
The DP-chromatic number of a bidirected graph is equal to the DP-chromatic number of its underlying graph.
\end{corollary}

DP-colorings are of special interest because they constitute a generalization of list-colorings: let $D$ be a digraph, let $C$ be a color set, and let $L:V(D) \to 2^C$ be a list-assignment. We define a cover $(X,H)$ of $D$ as follows: let $X_v=\{v\} \times L(v)$ for all $v \in V(D)$,  $V(H)=\bigcup_{v \in V(D)}X_v$, and $A(H)=\{(v,c)(v',c')~|~vv' \in A(D) \text{ and } c=c'\}$. It is obvious that $(X,H)$ indeed is a cover of $D$. Moreover, if $\varphi$ is an $L$-coloring of $D$, then $T=\{(v,\varphi(v))~|~ v \in V(D)\}$ is an acyclic transversal of $(X,H)$. On the other hand, given an acyclic transversal $T=\{(v_1,c_1),\ldots,(v_n,c_n)\}$ of $H$, we obtain an $L$-coloring of $D$ by coloring the vertex $v_i$ with $c_i$ for $i \in \{1,2,\ldots,n\}$. Thus, finding an $L$-coloring of $D$ is equivalent to finding an acyclic transversal of $(X,H)$. Hence, the \textbf{list-chromatic number} $\chi_\ell$ of $D$, which is the smallest integer $k$ such that $D$ admits an $L$-coloring for every list-assignment $L$ satisfying $|L(v)|\geq k$ for all $v \in V(D)$, is always at most the DP-chromatic number $\chi_{\text{DP}}(D)$. Moreover, by using a sequential coloring algorithm it is easy to verify that $\chi_{\text{DP}}(D) \leq \max\{\Delta^+(D),\Delta^-(D)\} + 1$. Hence, we obtain the following sequence of inequalities:
$$\chi(D) \leq \chi_\ell(D) \leq \chi_{\text{DP}}(D) \leq \max\{\Delta^+(D),\Delta^-(D)\} + 1.$$

\subsection{DP-Degree Colorable Digraphs}

We say that a digraph $D$ is \textbf{DP-degree colorable} if $D$ is $(X,H)$-colorable whenever $(X,H)$ is a cover of $D$ such that $|X_v| \geq \max\{d_D^+(v), d_D^-(v)\}$ for all $v \in V(D)$. In the following, we will give a characterization of the non DP-degree-colorable digraphs as well as a characterization of the edge-minimal corresponding 'bad' covers (see Theorem~\ref{theorem:main-result}). Clearly, it suffices to do this only for connected digraphs. For undirected graphs, those characterizations were given by Kim and Ozeki \cite{KiOz17}; for hypergraphs it was done by Schweser~\cite{Schwes18}. 

A \textbf{feasible configuration} is a triple $(D,X,H)$ consisting of a connected digraph $D$ and a cover $(X,H)$ of $D$. A feasible configuration $(D,X,H)$ is said to be \textbf{degree-feasible} if $|X_v| \geq \max\{d^+_D(v), d^-_D(v)\}$ for each vertex $v \in V(D)$. Furthermore, $(D,X,H)$ is \textbf{colorable} if $D$ is $(X,H)$-colorable, otherwise it is called \textbf{uncolorable}. The next proposition lists some basic properties of feasible configurations; the proofs are straightforward and left to the reader.

\begin{proposition} \label{prop_feas-config}
Let $(D,X,H)$ be a feasible configuration. Then, the following statements hold. 
\begin{itemize}
\item[\upshape (a)] For every vertex $v \in V(D)$ and every vertex $x \in X_v$, we have $d_H^+(x) \leq d_D^+(v)$ and $d_H^-(x) \leq d_D^-(v)$.
\item[\upshape (b)] Let $H'$ be a spanning subdigraph of $H$. Then, $(D,X,H')$ is a feasible configuration. If $(D,X,H)$ is colorable, then $(D,X,H')$ is colorable, too. Furthermore, $(D,X,H)$ is degree-feasible if and only if $(D,X,H')$ is degree-feasible.
\end{itemize}
\end{proposition}

The above proposition leads to the following concept. We say that a feasible configuration $(D,X,H)$ is \textbf{minimal uncolorable} if $(D,X,H)$ is uncolorable, but $(D,X,H-a)$ is colorable for each arc $a \in A(H)$. As usual, $H-a$ denotes the digraph obtained from $H$ by deleting the arc $a$. Clearly, if $|D| \geq 2$ and if $\tilde{H}$ is the arcless spanning digraph of $H$, then $(D,X,\tilde{H})$ is colorable. Thus, it follows from the above Proposition that if $(D,X,H)$ is an uncolorable feasible configuration, then there is a spanning subdigraph $H'$ of $H$ such that $(D,X,H')$ is a minimal uncolorable feasible configuration.

In order to characterize the class of minimal uncolorable degree-feasible configurations, we first need to introduce three basic types of degree-feasible configurations.

We say that $(D,X,H)$ is a \textbf{K-configuration} if $D$ is a complete digraph of order $n$ for some $n \geq 1$, and $(X,H)$ is a cover of $D$ such that the following conditions hold:
\begin{itemize}
\item $|X_v|=n-1$ for all $v \in V(D)$,
\item for each $v \in V(D)$ there is a labeling $x_v^1,x_v^2,\ldots,x_v^{n-1}$ of the vertices of $X_v$ such that $H^i=H[\{x_v^i~|~ v  \in V(D)\}]$ is a complete digraph for $i \in \{1,2,\ldots,n-1\}$, and
\item $H=H^1 \cup H^2 \cup \ldots \cup H^{n-1}$.
\end{itemize}
An example of a K-configuration with $n=4$ is given in Figure~\ref{fig_directed-config}. It is an easy exercise to check that each $K$-configuration is a minimal uncolorable degree-feasible configuration. Note that for $|D|=1$, we have $X_v=\varnothing$ for the only vertex $v \in V(D)$ and $H=\varnothing$ (and so there is no transversal of $(X,H)$).

We say that $(D,X,H)$ is a \textbf{C-configuration} if $D$ is a directed cycle of length $n \geq 2$ and $(X,H)$ is a cover such that $X_v=\{x_v\}$ for all $v \in V(D)$ and $A(H)=\{x_vx_u ~|~ vu \in A(D)\}$. Note that in this case, $H$ is a copy of $D$. Clearly, each C-configuration is a minimal uncolorable degree-feasible configuration.

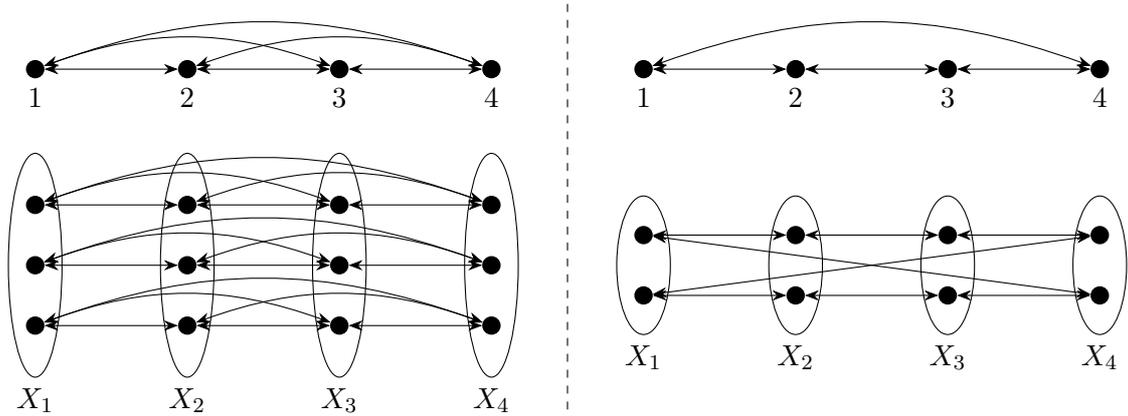
\begin{figure}[htp]

\resizebox{\linewidth}{!}{
\input{k_config_dir}}

\caption{A K-configuration and a BC-configuration for digraphs}
\label{fig_directed-config}
\end{figure}

We say that $(D,X,H)$ is an \textbf{odd BC-configuration} if $D$ is a bidirected cycle of odd length $\geq 5$ and $(X,H)$ is a cover of $D$ such that the following conditions are fulfilled:
\begin{itemize}
\item $|X_v|=2$ for all $v \in V(D)$,
\item for each $v \in V(D)$ there is a labeling $x_v^1,x_v^2$ of the vertices of $X_v$ such that $A(H)=\{x_v^ix_w^i ~|~ vw \in A(D) \text{ and } i \in \{1,2\}\}.$
\end{itemize}
Note that $H^i=H[\{x_v^i ~|~ v \in V(D)\}]$ is a bidirected cycle in $H$ and $H=H^1 \cup H^2$. It is easy to verify that every odd BC-configuration is a minimal uncolorable degree-feasible configuration.

We call $(D,X,H)$ an \textbf{even BC-configuration} if $D$ is a bidirected cycle of even length $\geq 4$, $(X,H)$ is a cover of $D$, and there is an arc $uu' \in A(D)$ such that:
\begin{itemize}
\item $|X_v|=2$ for all $v \in V(D)$,
\item for each $v \in V(D)$ there is a labeling $x_v^1,x_v^2$ of the vertices of $X_v$ such that $A(H)=\{x_v^ix_w^i ~|~ \{v,w\} \neq \{u,u'\},vw \in A(D),  \text{ and } i \in \{1,2\}\} \cup \{x_u^1x_{u'}^2, x_u^2x_{u'}^1, x_{u'}^2x_u^1, x_{u'}^1x_u^2\}$
\end{itemize}
Again, it is easy to check that every even BC-configuration is a minimal uncolorable degree-feasible configuration. By a \textbf{BC-configuration} we either mean an even or an odd BC-configuration.

Our aim is, to show that we can construct every minimal uncolorable degree-feasible configuration from the three basic configurations by using the following operation. Let $(D^1,X^1,H^1)$ and $(D^2,X^2,H^2)$ be two feasible configurations, which are \textbf{disjoint}, that is, $V(D^1) \cap V(D^2) = \varnothing$ and $V(H^1) \cap V(H^2)= \varnothing$. Furthermore, let $D$ be the digraph obtained from $D^1$ and $D^2$ by identifying two vertices $v^1 \in V(D^1)$ and $v^2 \in V(D^2)$ to a new vertex $v^*$. Finally, let $H=H^1 \cup H^2$ and let $X:V(D) \to 2^{V(H)}$ be the mapping such that 
$$X_v=\begin{cases}
X^1_{v^1} \cup X^2_{v^2} & \text{if } v=v^*,\\
X_v^i & \text{if } v \in V(D^i) \setminus \{v^i\} \text{ and } i \in \{1,2\}
\end{cases}$$
for $v \in V(H)$. Then, $(D,X,H)$ is a feasible configuration and we say that $(D,X,H)$ is obtained from $(D^1,X^1,H^1)$ and $(D^2,X^2,H^2)$ by \textbf{merging} $v^1$ and $v^2$ to $v^*$.  

Now we define the class of \textbf{constructible configurations} as the smallest class of feasible configurations that contains each K-configuration, each C-configuration and each BC-configuration and that is closed under the merging operation. We say that a digraph is a \textbf{DP-brick} if it is either a complete digraph, a directed cycle, or a bidirected cycle. 
Thus, if $(D,X,H)$ is a constructible configuration, then each block of $D$ is a DP-brick. The next proposition is straightforward and left to the reader.

\begin{proposition} \label{prop_block-structure}
Let $(D,X,H)$ be a constructible configuration. Then, for each block $B \in \mathcal{B}(D)$ there is a uniquely determined cover $(X^B,H^B)$ of $B$ such that the following statements hold:
\begin{itemize}
\item[\upshape (a)] For each block $B \in \mathcal{B}(D)$, the triple $(B,X^B,H^B)$ is a {\upshape K}-configuration, a {\upshape C}-configuration, or a {\upshape BC}-configuration.
\item[\upshape (b)] The digraphs $H^B$ with $B \in \mathcal{B}(D)$ are pairwise disjoint and $H = \bigcup_{B \in \mathcal{B}(D)} H^B$.
\item[\upshape (c)] For every vertex $v$ from $V(D)$ we have $X_v = \displaystyle{\bigcup_{B \in \mathcal{B}(D), v \in V(B)}X_v^B}$.
\end{itemize}
\end{proposition}

Our aim is to prove that the class of constructible configurations and the class of minimal uncolorable degree-feasible configurations coincide. This leads to the following theorem.

\begin{theorem} \label{theorem:main-result}
Suppose that $(D,X,H)$ be a degree-feasible configuration. Then, $(D,X,H)$ is minimal uncolorable if and only if $(D,X,H)$ is constructible.
\end{theorem}

For DP-colorings of undirected graphs, an analogous result was proven by Bernshteyn, Kostochka and Pron in \cite{BeKoPro17}. However, they only characterized the graphs that are not DP-degree colorable, but not the corresponding bad covers. This was done later by Kim and Ozeki \cite{KiOz17}. The third author of this paper extended the characterization of the non DP-degree colorable graphs to hypergraphs  \cite{Schwes18} and characterized also the minimal uncolorable degree-feasible configurations; since he used the same terminology as we do and since we need to refer to the undirected version in our proof, we only state the part of his theorem examining simple undirected graphs.

Regarding undirected graphs, a \textbf{degree-feasible} configuration is a triple $(G,X,H)$, where $G$ is an undirected (simple) graph and $(X,H)$ is a cover of  $G$ such that $|X_v| \geq d_G(v)$ for all $v \in V(G)$. A degree-feasible configuration $(G,X,H)$ is \textbf{colorable} if $G$ is $(X,H)$-colorable, otherwise it is called \textbf{uncolorable}. Moreover, $(G,X,H)$ is \textbf{minimal uncolorable} if  $(G,X,H)$ is uncolorable but $(G,X,H-e)$ is colorable for each edge $e \in E(H)$. Furthermore,  for undirected graphs, the definition of a \textbf{K-configuration} and a \textbf{BC-configuration} can be deduced  from the above definition for digraphs by considering the underlying undirected graphs (see Figure~\ref{fig_undirected-config}). Finally, for undirected graphs we define the class of constructible configurations as the smallest class of configurations that contains each K-configuration and each BC-configuration and that is closed under the merging operation. The proof of the following theorem can be found in \cite{Schwes18}.

\begin{figure}[htp]
\resizebox{\linewidth}{!}{
\input{k_c_config}}

\caption{A K-configuration and a BC-configuration for undirected graphs }
\label{fig_undirected-config}
\end{figure}
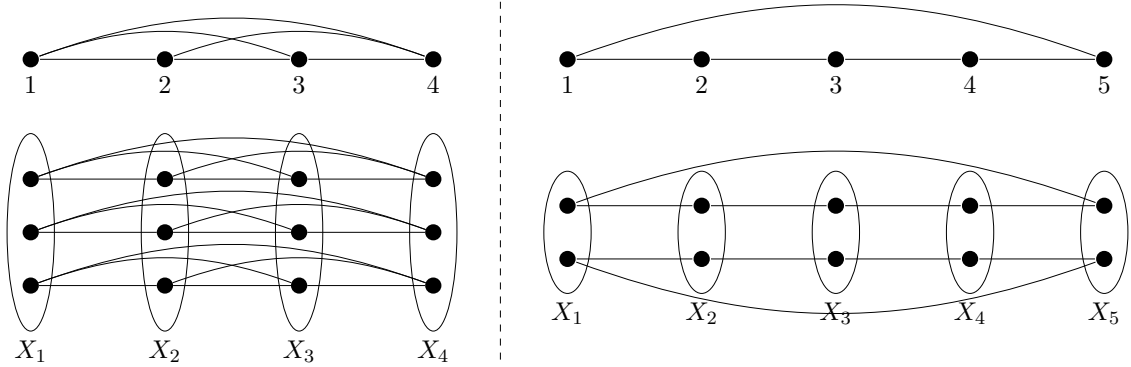

\begin{theorem}\label{theorem_undirected-case}
Let $G$ be a simple graph and let $(G,X,H)$ be a degree-feasible configuration. Then, $(G,X,H)$ is minimal uncolorable if and only if $(G,X,H)$ is constructible.
\end{theorem} 

In the following, given a feasible configuration $(D,X,H)$, we will often fix a vertex $v \in V(D)$ and regard the feasible configuration $(D',X',H')$, where $D'=D-v$, $X'$ is the restriction of $X$ to $V(D) \setminus \{v\}$ and $H'=H-X_v$. For the sake of readability, we will write $(X',H')=(X,H)/v$. 

First we state some important facts about minimal uncolorable degree-feasible configurations. Those will lead to powerful tools and operations that we use in order to characterize the minimal uncolorable degree-feasible configurations. Recall that the digraph $D$ of a degree-feasible configuration $(D,X,H)$ is connected by definition.

\begin{proposition}
\label{prop_main-prop}
Let $(D,X,H)$ be a degree-feasible configuration. If $(D,X,H)$ is uncolorable, then the following statements hold:
\begin{itemize}
\item[\upshape (a)] $|X_v|=d_D^+(v)=d_D^-(v)$ for all $v \in V(D)$. As a consequence, $D$ is Eulerian.
\item[\upshape (b)] Let $v \in V(D)$ and let $(X',H')=(X,H)/v$. Then, there is an acyclic transversal of $(X',H')$.
\item[\upshape (c)] Let $v \in V(D)$ and let $T$ be an acyclic transversal of $(X',H')=(X,H)/v$. Moreover, let $T^+=\bigcup_{u \in N_D^+(v)}(X_u \cap T)$ and let $T^-=\bigcup_{u \in N_D^-(v)}(X_u \cap T)$. Then, the arcs from $E_H(X_v,T^+)$ form a perfect matching in $H[X_v \cup T^+]$ and the arcs from $E_H(T^-,X_v)$ form a perfect matching in $H[X_v \cup T^-]$.
\end{itemize}
\end{proposition}

\begin{proof}
(a) The proof is by induction on the order of $D$. The statement is clear if $|D|=1$ as in this case $X_v=\varnothing$ for the only vertex $v$ of $D$. Now assume that $|D| \geq 2$. By assumption, $|X_v| \geq \max\{d_D^+(v),d_D^-(v)\}$ for all $v \in V(D)$. Hence, it suffices to show  $|X_v| \leq \min\{d_D^+(v),d_D^-(v)\}$ for all $v \in V(D)$. Suppose, to the contrary, that there is a vertex $v \in V(D)$ with $|X_v| > \min\{d_D^+(v), d_D^-(v)\}$, say $|X_v| > d_D^-(v)$ (by symmetry). Let $D'=D-v$ and let $(X',H')=(X,H)/v$. We claim that $D'$ is not $(X',H')$-colorable. Otherwise, there would be an acyclic transversal $T$ of $(X',H')$. As $|X_v| > d_D^-(v)$ it follows from (C2) that there is a vertex $x \in X_v$ such that $x'x \not \in A(H)$ for all $x' \in T'$. Consequently, $T \cup \{x\}$ is an acyclic transversal of $(X,H)$ as $x$ has no in-neighbor in $H[T \cup \{x\}]$, that is, $(D,X,H)$ is colorable, a contradiction. Thus, $D'$ is not $(X',H')$-colorable, as claimed. Hence, $D'$ contains a connected component $D''$ such that $(D'',X'',H'')$ is uncolorable,  where $X''$ is the restriction of $X'$ to $V(D'')$ and $H''=H'[\bigcup_{v \in V(D'')}X_v]$. By applying the induction hypothesis to $(D'',X'',H'')$ we conclude that $|X_w|=d_{D''}^+(w)=d_{D''}^-(w)$ for all $w \in D''$. As $D$ is connected, there is a vertex $w \in D''$ that is adjacent to $v$ in $D$. By symmetry, we may assume $wv \in A(D)$. But then, $$d_{D''}^+(w)=|X_w| \geq \max\{d_D^+(w), d_D^-(w)\} \geq d_{D''}^+(w) + 1,$$ which is impossible. This proves (a).

(b) For this proof, let $D'=D-v$ and let $(X',H')=(X,H)/v$. Let $D''$ be an arbitrary component of $D'$, let $X''$ be the restriction of $X'$ to $V(D'')$, and let $H''=H[\bigcup_{u \in V(D'')}X_u]$. Then, $(D'', X'',H'')$ is a degree-feasible configuration. As $D$ is connected, there is at least one vertex $u \in V(D'')$ that is in $D$ adjacent to $v$, say $uv \in A(D)$. By (a), this implies $|X_u| = d_D^+(u) > d_{D''}^+(u)$. Again by (a), we conclude that $(D'',X'',H'')$ is colorable, \emph{i.e.}, $(X'',H'')$ admits an acyclic transversal $T_{D''}$. Let $T$ be the union of the sets $T_{D''}$ over all components $D''$ of $D - v$. Then, $T$ is an acyclic transversal of $(X',H')$.
%

(c) For the proof, we first assume that there is a vertex $x \in X_v$ such that no vertex of $T$ is an out-neighbor of $x$ in $H$. Then, similarly to the proof of (a), we conclude that $T \cup \{x\}$ is an acyclic transversal of $(X,H)$, a contradiction. Hence, each vertex $x \in X_v$ has in $H$ at least one out-neighbor belonging to $T$. Moreover, for each vertex $u \in N_D^+(v)$ and for the unique vertex $x' \in T \cap X_{u}$ there may be at most one vertex $x \in X_v$ with $xx' \in A(H)$ (by (C2)). As $|X_v|=d_{D}^+(v)=|N_D^+(v)|$, this implies that for each vertex $x \in X_v$ there is exactly one vertex $x' \in T$ with $xx' \in A(H)$. Thus, the arcs from $X_v$ to $T^+=\bigcup_{u \in N_D^+(v)}(X_u \cap T)$ are a perfect matching in $H[X_v \cup T^+]$ as claimed. Using a similar argument, it follows that $E_H(T^-,X_v)$ is a perfect matching in $H[X_v \cup T^-]$.
\end{proof}

The above proposition is our main tool in order to characterize the minimal uncolorable degree-feasible configurations. The next proposition shows the usefulness of the merging operation.

\begin{proposition}\label{prop_merging}
Let $(D^1,X^1,H^1)$ and $(D^2,X^2,H^2)$ be two disjoint feasible configurations, and let $(D,X,H)$ be the configuration that is obtained from $(D^1,X^1,H^1)$ and $(D^2,X^2,H^2)$ by merging two vertices $v^1 \in V(D^1)$ and $v^2 \in V(D^2)$ to a new vertex $v^*$. Then, $(D,X,H)$ is a feasible configuration and the following statements are equivalent:
\begin{itemize}
\item[\upshape (a)] Both $(D^1,X^1,H^1)$ and $(D^2,X^2,H^2)$ are minimal uncolorable degree-feasible configurations.
\item[\upshape  (b)] $(D,X,H)$ is a minimal uncolorable degree-feasible configuration.
\end{itemize}
\end{proposition}

\begin{proof}
First we show that (a) implies (b). Clearly, $(D,X,H)$ is degree-feasible. Assume that $(D,X,H)$ is colorable. Then, there is an acyclic transversal $T$ of $(X,H)$.
As $X_{v^*}=X_{v^1} \cup X_{v^2}$, this implies that at least one of $v_1$ and $v_2$ (by symmetry, we can assume it is $v_1$) observes $|T \cap X_{v^1}|=1$. Thus, $T^1=T \cap V(H^1)$ is an acyclic transversal of $(X^1,H^1)$ and so $(D^1,X^1,H^1)$ is colorable, a contradiction to (a). This proves that $(D,X,H)$ is uncolorable. Now let $a \in A(H)$ be an arbitrary arc. By symmetry, we may assume $a \in A(H^1)$. Since $(D^1,X^1,H^1)$ is minimal uncolorable, there is an acyclic transversal $T^1$ of $(X^1,H^1-a)$. Since $(D^2,X^2,H^2)$ is also uncolorable and degree-feasible, there is an acyclic transversal $T^2$ of  $(X^2,H^2)/v^2$ (by Proposition~\ref{prop_main-prop}(b)). However, as $H=H^1 \cup H^2$ and $H_1 \cap H_2 = \varnothing$, the set $T=T^1 \cup T^2$ is an acyclic transversal of $(X,H-a)$ and so $(D,X,H-a)$ is colorable. Thus, (b) holds.

To prove that (b) implies (a), we first show that $(D^1,X^1,H^1)$ is a minimal uncolorable. Assume that $(D^1,X^1,H^1)$ is colorable, that is, $(X^1,H^1)$ has an acyclic transversal $T^1$. Since $(D,X,H)$ is an uncolorable degree-feasible configuration and as $H^2-v^2$ is a proper subdigraph of $H-v^*$, there is an acyclic transversal $T^2$ of $(X^2,H^2)/v^2$ (by Proposition~\ref{prop_main-prop}(b)). Then again, $T=T^1 \cup T^2$ is an acyclic transversal of $(X,H)$, contradicting (b). Thus, $(D^1,X^1,H^1)$ is uncolorable. Now let $a \in A(H^1)$ be an arbitrary arc. Then, as $(D,X,H)$ is minimal uncolorable and $a \in A(H)$, there is an acyclic transversal $T$ of $(X,H-a)$ and $T^1=T \cap V(H^1)$ clearly is an acyclic transversal of $(X^1,H^1-a)$. Consequently, $(D^1,X^1,H^1-a)$ is colorable. This shows that $(D^1,X^1,H^1)$ is minimal uncolorable. By symmetry $(D^2,X^2,H^2)$ is minimal uncolorable, too.

It remains to show that $(D^j,X^j,H^j)$ is degree-feasible for $j \in \{1,2\}$. As $(D,X,H)$ is an uncolorable degree-feasible configuration, Proposition~\ref{prop_main-prop}(a) implies that  
\begin{align}\label{eq_degree-prop10}
|X_v|=d_D^+(v)=d_D^-(v) \text{ for all  } v \in V(D).
\end{align} Consequently, each vertex from $D^j-v^j$ is eulerian in $D^j$. Since $$\sum_{u \in V(D^j)} d_{D^j}^+(u)=\sum_{u \in V(D^j)}d_{D^j}^-(u)=|A(D^j)|$$ is the number of arcs of $D^j$, it follows that $d_{D^j}^+(v^j)=d_{D^j}^-(v^j)$, and so $D^j$ is Eulerian for $j\in \{1,2\}$. Moreover, it follows from ~\eqref{eq_degree-prop10} that $|X_v|=d_D^+(v)=d_{D^j}^+(v)=d_{D^j}^-(v)$ for all $v \in V(D^j) \setminus \{v^j\}$ and $j \in \{1,2\}$.  If $|X_{v^j}|<d_D^+(v^j)$ for some $j \in \{1,2\}$, then $|X_{v^{3-j}}|>d_D^+(v^{3-j})$ and so $(D^{3-j},X^{3-j},H^{3-j})$ would be colorable by Proposition~\ref{prop_main-prop}(a), a contradiction. Hence, $(D^j,X^j,H^j)$ is degree-feasible for $j \in \{1,2\}$.
\end{proof}


In order to prove Theorem~\ref{theorem:main-result}, we need some more tools. The first one, which will be frequently used in the following, is the so-called \textbf{shifting operation}. Let $(D,X,H)$ be a minimal uncolorable degree-feasible configuration, let $D'=D-v$ for some $v \in V(D)$, and let $T$ be an acyclic transversal of $(X',H')=(X,H)/v$ (which exists by Proposition~\ref{prop_main-prop}(b)). Then it follows from Proposition~\ref{prop_main-prop}(c) that for each vertex $x \in X_v$ there is exactly one vertex $x' \in T$ with $xx' \in A(H)$ and exactly one vertex $x'' \in T$ with $x''x \in A(H)$. Let $v'$ and $v''$ be the vertices from $V(D)$ such that $x' \in X_{v'}$ and $x'' \in X_{v''}$. Then, $T'=T \setminus \{x'\} \cup \{x\}$ and $T''=T \setminus \{x''\} \cup \{x\}$ are acyclic transversals of $(X,H)/v'$ and $(X,H)/v''$, respectively, since in $H[T']$ (respectively $H[T'']$) the vertex $x$ has no out-neighbor (respectively no in-neighbor) and, hence, $x$ cannot be contained in a directed cycle. We say that $T'$ (respectively $T''$) evolves from $T$ by \textbf{shifting} the color $x'$ (respectively $x''$) to $x$. Of course, the shifting operation may be applied repeatedly. The next proposition can be easily deduced from Proposition~\ref{prop_main-prop} by applying the shifting operation. The statements of the proposition are illustrated in Figure~\ref{fig_forbidden-config}.

\begin{proposition} \label{prop_forbidden_config}
Let $(D,X,H)$ be a minimal uncolorable degree-feasible configuration, let $v \in V(D)$, and let $T$ be an acyclic transversal of $(X',H')=(X,H)/v$. Then, the following statements hold:
\begin{itemize}
\item[\upshape (a)] For every vertex $x \in X_v$ we have $|N_H^+(x) \cap T|=1$ and $|N_H^-(x) \cap T|=1$. 
\item[\upshape (b)] Let $u \in N_D^+(v)$ and let $X_u \cap T = \{x_u\}$. Then, there is a vertex $x \in X_v$ such that $xx_u \in A(H)$ and $N_H^-(x_u) \cap T=\varnothing$.
\item[\upshape (c)] Let $w \in N_D^-(v)$ and let $X_w \cap T = \{x_w\}$. Then, there is a vertex $x \in X_v$ such that $x_wx \in A(H)$ and $N_H^+(x_w) \cap T=\varnothing$.
\end{itemize}
\end{proposition}

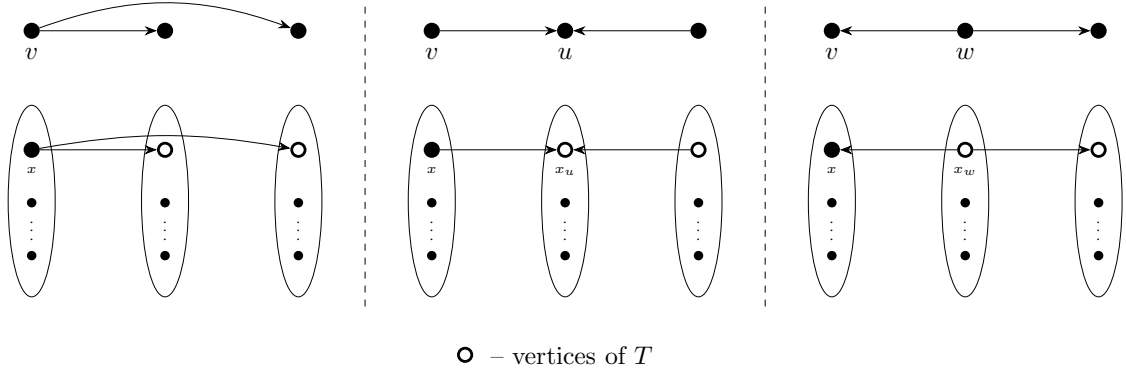
\begin{figure}[htp]
\centering
\resizebox{\linewidth}{!}{
\input{forbidden_config}
}
\caption{Forbidden configurations for $(D,X,H)$.}
\label{fig_forbidden-config}
\end{figure}

\begin{proof}
Statement (a) is a direct consequence of Proposition~\ref{prop_main-prop}(c). In order to prove (b) let $u \in N_D^+(v)$ and let $X_u \cap T=\{x_u\}$. Again from Proposition~\ref{prop_main-prop}(c) it follows that there is a vertex $x \in X_v$ with $xx_u \in A(H)$. Now assume that there is a vertex $x' \in N_H^-(x_u) \cap T$. Let $T'$ be the transversal of $(X,H)/u$ that evolves from $T$ by shifting $x_u$ to $x$. Then, both $x'$ and $x$ are in-neighbors of $x_u$ in $H$ and so $|N_H^-(x_u) \cap T'| \geq 2$, a contradiction to (a). This proves (b). By symmetry, (c) follows. 
\end{proof}

\begin{proposition}\label{prop_oppositearcs->bidirected}
Let $(D,X,H)$ be a  minimal uncolorable  degree-feasible configuration and let $u,v \in V(D)$ such that there are opposite arcs between $u$ and $v$. Then, $H[X_u \cup X_v]$ is bidirected.
\end{proposition}

\begin{proof}
Suppose, the statement is false. Then there are vertices $x_u \in X_u$ and $x_v \in X_v$ with $x_ux_v \in A(H)$ and $x_vx_u \not \in A(H)$. Since $(D,X,H)$ is minimal uncolorable, there is an acyclic transversal $T$ of $(X,H-x_ux_v)$. Furthermore, $T$ must contain both $x_u$ and $x_v$ as otherwise $T$ would be an acyclic transversal of $(X,H)$, a contradiction. Then, $T'=T \setminus \{x_v\}$ is an acyclic transversal of $(X',H')=(X,H)/v$. As $u \in N_D^-(v)$, it follows from Proposition~\ref{prop_forbidden_config}(b) that there is a vertex $x \in X_v$ with $xx_u \in A(H)$.  Since $x_vx_u \not \in A(H)$, $x \neq x_v$. Let $T^*$ be the transversal that evolves from $T'$ by shifting $x_u$ to $x_v$. Then, $x_u$ has an in-neighbor $x^*$ from $T^*$ in $H$ (by Proposition~\ref{prop_forbidden_config}(a)) and $x^* \not \in X_v$ (since $x_vx_u \not \in A(H)$). Moreover, $x^*$ is contained in the transversal $\tilde{T}$ that evolves from $T'$ by shifting $x_u$ to $x$ and so $\{x,x^*\} \subseteq N_H^-(x_u) \cap \tilde{T}$. Consequently, $|N_H^-(x_u) \cap \tilde{T}| > 1$, which contradicts Proposition~\ref{prop_forbidden_config}(a). Hence $x=x_v$, and so $x_vx_u \in A(H)$, a contradiction. 
\end{proof}

In particular, the above proposition implies the following concerning the shifting operation. Let $(D,X,H)$ be a minimal uncolorable degree-feasible configuration, let $v \in V(D)$ and let $T$ be an acyclic transversal of $(X',H')=(X,H)/v$ (which exists by Proposition~\ref{prop_main-prop}(b)). Then it follows from the above proposition together with Proposition~\ref{prop_forbidden_config}(b)(c) that for each vertex $u$ that is in $D$ adjacent to $v$ and for the unique vertex $x_u \in X_u \cap T$ there is exactly one vertex $x_v \in X_v$ that is in $H$ adjacent to $x_u$. Hence, $x_v$ is the unique vertex from $X_v$ to which we can shift the color $x_u$. Thus, in the following we may regard the shifting operation as an operation in the digraph $D$ rather than in $H$ and write $\mathbf{u \to v}$ in order to express that we shift the color from the corresponding vertex $x_u$ to $x_v$.

As another consequence of Proposition~\ref{prop_oppositearcs->bidirected} we easily obtain the following corollary.

\begin{corollary}\label{prop_Dbidirected->Hbidirected}
Let $(D,X,H)$ be a degree-feasible minimal uncolorable configuration such that $D$ is bidirected. Then $H$ is bidirected, too.
\end{corollary}

Having all those tools available, we are finally ready to prove our main theorem.

\subsection{Proof of Theorem~\ref{theorem:main-result}}

This subsection is devoted to the proof of Theorem~\ref{theorem:main-result}, which we recall for convenience.

\setcounter{theorem}{6}
\begin{theorem} \label{theorem:main-result}
Suppose that $(D,X,H)$ is a degree-feasible configuration. Then, $(D,X,H)$ is minimal uncolorable if and only if $(D,X,H)$ is constructible.
\end{theorem}

\begin{proof}
If $(D,X,H)$ is constructible, then $(D,X,H)$ is minimal uncolorable (by Proposition~\ref{prop_merging} and as each K-, C-, and BC-configuration is a minimal uncolorable degree-feasible configuration). 

Now let $(D,X,H)$ be a minimal uncolorable degree-feasible configuration. We prove that $(D,X,H)$ is constructible by induction on the order of $D$. If $|D|=1$, then $V(D)=\{v\}$, $X_v=\varnothing$ and $H=\varnothing$ and so $(D,X,H)$ is a K-configuration. Thus, we may assume that $|D| \geq 2$.  By Proposition~\ref{prop_main-prop}(a), 
\begin{align}\label{eq_eulerian}
|X_v|=d_D^+(v)=d_D^-(v) \quad \text{for all }v \in V(D).
\end{align}
We distinguish between two cases.

\case{1}{$D$ contains a separating vertex $v^*$.} Then, $D$ is the union of two connected induced subdigraphs $D^1$ and $D^2$ with $V(D^1) \cap V(D^2) = \{v^*\}$ and $|D^j| < |D|$ for $j \in \{1,2\}$. By equation \eqref{eq_eulerian}, all vertices from $D^j$ except from $v^*$ are Eulerian in $D^j$ (for $j \in \{1,2\}$). However, since $$\sum_{u \in V(D^j)} d_{D^j}^+(u)=\sum_{u \in V(D^j)}d_{D^j}^-(u)=|A(D^j)|$$ is the number of arcs of $D^j$, it follows that $d_{D^j}^+(v^*)=d_{D^j}^-(v^*)$ and so $D^j$ is Eulerian for $j \in \{1,2\}$. For $j \in \{1,2\}$, by $\mathcal{T}^j$ we denote the set of all subsets $T$ of $H$ with $|T \cap X_v|=1$ for all $v \in V(D^j)$ and $|T \cap X_u|=0$ for all $u \in V(D^{3-j})\setminus\{v^*\}$ such that $H[T]$ is acyclic. As $(D,X,H)$ is uncolorable and degree-feasible, both $\mathcal{T}^1$ and $\mathcal{T}^2$ are non-empty (by Proposition~\ref{prop_main-prop}(b)). Moreover, for $j \in \{1,2\}$, let $X_j$ be the set of all vertices of $X_{v^*}$ that do not occur in any set from $\mathcal{T}^j$. We claim that $X_{v^*}=X_1 \cup X_2$. For otherwise, there is a vertex $x \in X_{v^*} \setminus (X_1 \cup X_2)$. Then, $x$ is contained in two sets $T^1 \in \mathcal{T}^1$ and $T^2 \in \mathcal{T}^2$, and so $T=T^1 \cup T^2$ is an acyclic transversal of $(X,H)$. Thus, $(D,X,H)$ is colorable, a contradiction. Consequently, $X_{v^*}=X_1 \cup X_2$. For $j \in \{1,2\}$, we define a cover $(X^j,H^j)$ of $D^j$ as follows. For $v \in V(D^j)$, let
$$X_v^j=\begin{cases}
X_v & \text{if } v \neq v^* \\
X_j & \text{if } v = v^*,
\end{cases}$$
and let $H^j=H[\bigcup_{v \in V(D^j)} X_v^j]$. Then, $(D^j,X^j,H^j)$ is an uncolorable feasible configuration for $j \in \{1,2\}$: Suppose w.l.o.g. that $(D^1, X^1, H^1)$ has an acyclic transversal $T$. Then $T$ is in $\mathcal{T}^1$, but $T$ contains a vertex $x \in X_{v^*}^1=X_1$, which is impossible. Furthermore, for each vertex $v \in V(D^j) \setminus \{v^*\}$, equation \eqref{eq_eulerian} implies that $|X_v|=d_{D}^+(v)=d_{D^j}^+(v)$. As $(D^j,X^j,H^j)$ is uncolorable and $D^j$ is connected, it follows from Proposition~\ref{prop_main-prop}(a) that $|X_{v^*}^j| \leq d_{D^j}^+(v^*)$ for $j \in \{1,2\}$. Since $X_{v^*}=X_1 \cup X_2 = X^1_{v^*} \cup X^2_{v^*}$, we conclude from \eqref{eq_eulerian} that
$$|X_{v^*}^1| + |X_{v^*}^2| \geq |X_{v^*}^1 \cup X_{v^*}^2| = |X_{v^*}| = d_{D}^+(v^*) = d_{D^1}^+(v^*) + d_{D^2}^+(v^*),$$
and, thus, $|X_{v^*}^j|=d_{D^j}^+(v^*)(=d_{D^j}^-(v^*))$ and $X_{v*}^1 \cap X_{v*}^2 = \varnothing$. Consequently, $(D^j,X^j,H^j)$ is a degree-feasible configuration. Moreover, $H'=H^1 \cup H^2$ is a spanning subdigraph of $H$ and $V(H^1) \cap V(H^2) = \varnothing$. So, $(D,X,H')$ is a degree-feasible configuration that is obtained from two ismorphic copies of $(D^1,X^1,H^1)$ and $(D^2,X^2,H^2)$ by the merging operation. Clearly, $(D,X,H')$ is uncolorable. Otherwise, there would exist an acyclic transversal $T$ of $(X,H')$ and by symmetry we may assume that $T$ would contain a vertex of $X_{v^*}^1$. But then, $T^1=T \cap V(H^1)$ would be an acyclic transversal of $(X^1,H^1)$, contradicting that $(D^1,X^1,H^1)$ is uncolorable. As $(D,X,H)$ is minimal uncolorable and as $H'$ is a spanning subhypergraph of $H$, this implies that $H=H'$ and $(D,X,H)$ is obtained from two isomorphic copies of $(D^1,X^1,H^1)$ and $(D^2,X^2,H^2)$ by the merging operation. Then, by Proposition~\ref{prop_merging}, both $(D^1,X^1,H^1)$ and $(D^2,X^2,H^2)$ are minimal uncolorable. Applying the induction hypothesis leads to $(D^j,X^j,H^j)$ being constructible for $j \in \{1,2\}$, and so $(D,X,H)$ is constructible. Thus, the proof of the first case is complete.
\medskip

\case{2}{$D$ is a block.} Then, each vertex of $D$ is contained in a cycle of the underlying graph $G(D)$. We prove that $(D,X,H)$ is a K-, C- or BC-configuration by examining the cycles that may occur in $G(D)$ and showing that the cycles always imply that the structure of $(D,X,H)$ is as claimed. This is done via a sequence of claims. In the first three claims we analyze the case where $D$ contains a digon and show that in this case, both $D$ and $H$ are bidirected. Then, we can apply Theorem~\ref{theorem_undirected-case} to the undirected configuration $(G(D),X,G(H))$ in order to deduce that $(D,X,H)$ is a K- or BC-configuration. Afterwards, we analyze the case that $D$ does not contain any digons and prove that this implies that $(D,X,H)$ is a C-configuration. Recall that if $C$ is a cycle in the underlying graph $G(D)$, then $D_C$ is the maximum subdigraph of $D$ such that $G(D_C)=C$.
\begin{claim}\label{claim_c3}
Let $C$ be a cycle of length $3$ in the underlying graph $G(D)$. If $D_C$ is not a directed cycle,  then $V(C)$ induces a complete digraph in $D$.
\end{claim}

\begin{proof2}
Let $v_1,v_2,v_3$ be the vertices of $C$. By symmetry, assume that  $\{v_3v_1,v_1v_2,v_3v_2\} \subseteq A(D)$. We prove that $v_1v_3 \in A(D)$. Let $T$ be an acyclic transversal of $(X',H')=(X,H)/v_1$, let $x_j$ be the unique vertex from $X_{v_j} \cap T$ (for $j \in \{2,3\}$) and let $x_1 \in X_{v_1}$ such that $x_3x_1 \in A(H)$ (such a vertex exists by Proposition~\ref{prop_forbidden_config}(c)). Then, by Proposition~\ref{prop_forbidden_config}(c), $x_3x_2 \not \in A(H)$. Furthermore, by Proposition~\ref{prop_forbidden_config}(a), $x_1$ must have an out-neighbor $x$ in $T$. Assume that $x \in T \setminus \{x_2,x_3\}$. Then we can shift $v_3 \to v_1$, $v_2 \to v_3$ and $v_1 \to v_2$ and get a new acyclic transversal $T'$ of $(X',H')$. Moreover, if $x_2'$ is the vertex from $X_{v_2} \cap T'$, due to the shifting we have $x_1x_2' \in A(H)$. Since $T \setminus (X_{v_2} \cup X_{v_3}) = T' \setminus (X_{v_2} \cup X_{v_3})$ we conclude $N_H^+(x_1) \cap T' \supseteq \{x_2' , x \}$ and so $|N_H^+(x_1) \cap T'| \geq 2$, contradicting Proposition~\ref{prop_forbidden_config}(a) (see Figure~\ref{fig_1-claim1}). Hence, $x \in \{x_2,x_3\}$. If $x=x_2$ (and so $x_2'=x_2$), then starting from $T$ and then shifting $v_3 \to v_1$ and $v_2 \to v_3$ leads to an acyclic transversal $T^*$ of $(X,H)/v_2$ such that $|N_H^-(x_2)   \cap T^*| \geq 2$, in contradiction to Proposition~\ref{prop_forbidden_config}(a). Thus, $x=x_3$ and so $x_1x_3 \in A(H)$. However, this implies $v_1v_3 \in A(D)$ (by (C2)), as claimed. By symmetry we conclude that $D[V(C)]$ is a complete digraph and the proof is complete.
\end{proof2}
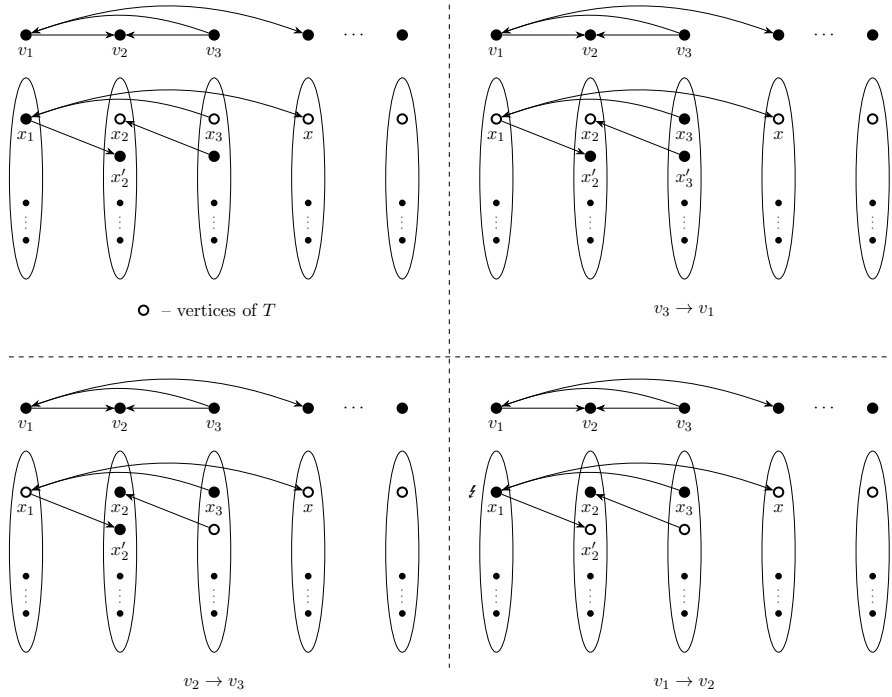
\begin{figure}[htp]
\centering
\resizebox{.81\linewidth}{!}{
\input{claim1-fig1}
}

\caption{$(D,X,H)$ before and after shifting $v_3 \to v_1, v_2 \to v_3$ and $v_1 \to v_2$.}
\label{fig_1-claim1}
\end{figure}

\begin{claim} \label{claim_inducedcycle} 
Let $C$ be an induced cycle in the underlying graph $G(D)$. If $D_C$ contains a digon, then $D_C$ is a bidirected cycle.
\end{claim}
\begin{proof2}
Assume, to the contrary, that $D_C$ is not bidirected. Then (by symmetry) we can choose a cyclic ordering $v_1,v_2,\ldots,v_p$ of the vertices of $C$ such that $v_1v_2, v_2v_1$ and $v_1v_p$ are arcs of $D$ and that $v_pv_1 \not \in A(D)$. Let $T$ be an acyclic transversal of $(X',H')=(X,H)/v_1$. For $i \in \{2,3,\ldots,p\}$ let $x_i$ be the vertex from $X_{v_i} \cap T$. By Proposition~\ref{prop_forbidden_config}(b) and Proposition~\ref{prop_oppositearcs->bidirected}, there is a vertex $x \in X_{v_1}$ that is joined to $x_2$ by opposite arcs and a vertex $x' \in X_{v_1}$ with $x'x_p \in A(H)$. Moreover, by Proposition~\ref{prop_forbidden_config}(a), $x \neq x'$. By shifting the vertices $v_2\to v_1,v_3 \to v_2, \ldots, v_p \to v_{p-1}$ counterclockwise on the cycle $C$ we obtain from Proposition~\ref{prop_forbidden_config}(c) that $x$ has an out-neighbor $x_p'$ in $X_p$. If we further shift $v_1 \to v_p$, we get a new acyclic transversal $T'$ of $(X',H')$ such that $x_p' \in T'$. By Proposition~\ref{prop_forbidden_config}(a), there must exist a vertex $y \in T'$ with $yx \in A(H)$. As $x_2$ is the unique in-neighbor of $x$ from $T$, since $v_1$ has no neighbors besides $v_2$ and $v_p$ from $V(C)$, and as the shifting only affected vertices from $C$, we conclude that $y \in X_{v_2} \cup X_{v_p}$. However, since $xx_p' \in A(H)$, it follows from Proposition~\ref{prop_forbidden_config}(a) that $x_2 \not \in T'$. Hence, $y \in X_{v_p}$ and so $v_pv_1 \in A(D)$, a contradiction.
\end{proof2}

\begin{claim}\label{claim_everything-bidir}
Suppose that $D$ contains a digon. Then, $D$ is bidirected.
\end{claim}
\begin{proof2}
Assume, to the contrary, that $D$ is not bidirected. As $D$ is a block this implies that in the underlying graph $G[D]$ there is a cycle $C$ of minimum length such that $D_C$ contains a digon but is not bidirected. Since $C$ has minimum length, we conclude that $C$ is an induced cycle of $G(D)$, but then it follows from Claim~\ref{claim_inducedcycle} that $D_C$ is bidirected, a contradiction. This proves the claim.
\end{proof2}

Suppose that $D$ contains at least one digon. Then, $D$ is bidirected (by Claim~\ref{claim_everything-bidir}) and it follows from Corollary~\ref{prop_Dbidirected->Hbidirected} that $H$ is bidirected, too. Consequently, $(G(D),X,G(H))$ is a degree-feasible configuration. Furthermore, an acyclic transversal of $(X,H)$ is an independent transversal of $(X,G(H))$ and vice versa, and it easy to check that $(G(D),X,G(H))$ is minimal uncolorable (as $(D,X,H)$ is minimal uncolorable). Then, as $G(D)$ is a block, it follows from Theorem~\ref{theorem_undirected-case} that $(G(D),X,G(H))$ is a K- or a BC-configuration. As a consequence,  $(D,X,H)$ is a K- or a BC-configuration and there is nothing left to show. Hence, from now on we may assume the following:
\begin{align} \label{eq_no-digon}
D~\text{does not contain a digon}. 
\end{align}
In the remaining part of the proof we will show that under the assumption \eqref{eq_no-digon}, the configuration $(D,X,H)$ is a C-configuration.

\begin{claim} \label{claim_K_4}
The underlying graph $G(D)$ does not contain any $K_4$.
\end{claim}
\begin{proof2}
Otherwise, $G(D)$ contains a cycle $C$ such that $D_C$ is not a directed cycle. Hence, by Claim~\ref{claim_c3}, $D$ would contain a complete digraph on three vertices, which contradicts \eqref{eq_no-digon}.
\end{proof2}

Recall that $K_4^-$ denotes the (undirected) graph that results from a $K_4$ by deleting any edge.

\begin{claim} \label{claim_K_^4-}
The underlying graph $G(D)$ does not contain any induced $K_4^-$.
\end{claim}

\begin{proof2}
Assume that $G(D)$ contains an induced $K_4^-$, say $\tilde{G}=G(\tilde{D})$. Then, by \eqref{eq_no-digon} and Claim~\ref{claim_c3}, $V(\tilde{D})=\{v_1,v_2,v_3,v_4\}$ and $A(\tilde{D})=\{v_1v_2,v_1v_3,v_2v_4,v_3v_4,v_4v_1\}$. Let $T$ be an acyclic transversal of $(X',H')=(X,H)/v_1,$ and for $i \in \{2,3,4\}$ let $x_i \in X_{v_i} \cap T$. Then it follows from Proposition~\ref{prop_forbidden_config}(b),(c) that there are vertices $x,x' \in X_{v_1}$ with $x'x_2 \in A(H)$ and $xx_3 \in A(H)$. By Proposition~\ref{prop_forbidden_config}(a), $x \neq x'$. By shifting $v_3 \to v_1$, we obtain that $x_4$ has an in-neighbor $x_3' \in X_{v_3}$ (by Proposition~\ref{prop_forbidden_config}(c)). We claim that $x'x_3' \in A(H)$. To see this, starting from $T$,  we can shift $v_3 \to v_1, v_4 \to v_3, v_2 \to v_4$ and then $v_1 \to v_2$ and obtain another acyclic transversal $T'$ of $(X',H')$ with $x_3' \in T'$. Then, $x'$ must have an out-neighbour $y$ in $T'$ (by Proposition~\ref{prop_forbidden_config}(a)). However, as $x \neq x'$, we deduce that $y \not \in X_{v_2}$. As we only shifted along vertices of $\tilde{D}$, we conclude that $y \not \in T' \setminus (X_2 \cup X_3 \cup X_4)$ (since otherwise $\{y,x_2\} \subseteq |N_H^+(x') \cap T|$, which leads to a contradiction to Proposition~\ref{prop_forbidden_config}(a)). Moreover, as $v_1v_4 \not \in A(D)$, this implies that $y \in X_{v_3}$ and so $y=x_3'$. Hence, $x'x_3' \in A(H)$, as claimed. But now, starting from $T$ we can shift $v_3 \to v_1, v_4 \to v_3$ and $v_1 \to v_4$ and obtain an acyclic transversal $T^*$ of $(X',H')$ that contains both $x_2$ and $x_3'$. As a consequence, $|N_H^+(x') \cap T^*|\geq2$, which contradicts Proposition~\ref{prop_forbidden_config}(a). This proves the claim.
\end{proof2}

\begin{claim}\label{claim_induced-cycle->cyclic}
Let $C$ be an induced cycle of the underlying graph $G(D)$. Then, $D_C$ is a directed cycle.
\end{claim}

\begin{proof2}
The proof is by reductio ad absurdum. Then, we can choose a cyclic ordering of the vertices of $C$, say $v_1,v_2,\ldots,v_p$, such that $\{v_1v_2, v_1v_p\}\subseteq A(D)$. Furthermore, let $T$ be an acyclic transversal of $(X',H')=(X,H)/v_1$ and, for $i \in \{1,2,\ldots,p\}$ let $x_i \in X_{v_i} \cap T$. Then, by Proposition~\ref{prop_forbidden_config}(a),(b), there are vertices $x \neq x'$ from $X_{v_1}$ with $xx_2 \in A(H)$ and $x'x_p \in A(H)$. Moreover, by shifting $v_p \to v_1, v_{p-1}\to v_p, \ldots, v_2 \to v_3$ clockwise around $C$, we obtain that $x'$ has an out-neighbor $x_2' \in X_{v_2}$ (by Proposition~\ref{prop_forbidden_config}(c)). We claim that $x_3x_2' \in A(H)$. Assume, to the contrary, that $x_3x_2' \not \in A(H)$ and let $T'$ be the transversal that results from $T$ by shifting $v_2 \to v_1$. Then, $x_2'$ must have an in-neighbor $y$ in $T'$ (by Proposition~\ref{prop_forbidden_config}(a)) and $y \not \in X_{v_i}$ for $i \in \{1,2,\ldots,p\}$ (as $x_3x_2' \not \in A(H)$, as $x' \not \in T'$ and as $C$ is an induced cycle). If instead, starting from $T$, we shift the vertices $v_p \to v_1, v_{p-1}v_p, \ldots, v_2 \to v_3$, we obtain an acyclic transversal $T^*$ of $(X,H)/v_2$ that contains both $x'$ as well as $y$, contradicting Proposition~\ref{prop_forbidden_config}(a) (as $x_2'$ has the two in-neighbors $x',y$ in $T^*$). Thus, $x_3x_2' \in A(H)$ and hence $v_3v_2 \in A(H)$.
As a consequence, there is also a vertex $x_3' \neq x_3$ from $X_{v_3}$ such that $x_3'x_2 \in A(H)$. Now we can shift $v_2 \to v_1$ and obtain an acyclic transversal of $(X,H)/v_2$. By repeating the same argumentation as above we conclude that $x_3'x_4 \in A(H)$.  Now, we can iterate this procedure for the remaining vertices of $C$ and obtain the following:
\begin{align}\label{eq_alternating}
\begin{gathered}
D_C \text{  is \textbf{alternating}, i.e. the vertices from } D_C \text{ alternatively have 
two}\\ \text{ in-neighbours and two out-neighbours in } D_C.
\end{gathered}
\end{align}

\noindent Note that this implies, in particular, that $C$ is even. Moreover, we conclude that for $i \in \{2,\ldots,p\}$ there are vertices $x_i \neq x_i'$ from $X_{v_i}$ such that the following holds:
\begin{itemize}
\item There is an acyclic transversal $T$ of $(X',H')=(X,H)/v_1$ that contains the vertices $x_2,x_3,\ldots,x_p$, and
\item $\{xx_2, x'x_2', xx_p', x'x_p \} \subseteq A(H)$ and for $i \in \{2,4,\ldots,p-2\}$ we have $x_{i+1}x_i',x_{i+1}'x_i \in A(H)$. 
\end{itemize}
Note that (beginning from $T$) by shifting $v_2 \to v_1, v_3 \to v_2, \ldots v_p \to v_{p-1}$ counterclockwise around $C$ and then shifting $v_1 \to v_p$ we obtain an acyclic transversal $T'$ of $(X',H')$ that contains the vertices $x_2',x_3',\ldots,x_p'$. 

Since $(D,X,H)$ is minimal uncolorable, $H[T \cup \{x\}]$ contains a directed cycle that must contain $x$, say $C_x$. Moreover, by Proposition~\ref{prop_forbidden_config}(a) and since $xx_2 \in A(H)$, $x$ and $x_2$ are consecutive on $C_x$. Let $z$ denote the vertex different from $x_2$ such that $x$ and $z$ are consecutive on $C_x$. Then, $z \not \in \{x_3,x_4,\ldots,x_p\}$. This is due to the fact that $C$ is an induced cycle in $G(D)$ (and so $v_1v_i \not \in A(D)$ for $i \in \{3,4,\ldots,p-1\}$) and that $xx_p' \in A(H)$ and, therefore, $xx_p \not \in A(H)$. Moreover, we obtain the following:
\begin{align} \label{eq_Cx-induced}
\begin{gathered}
C_x \text{ is an induced directed cycle of }  H[T \cup \{x\}] \text{ and}\\
\text{no vertex from } C_x  \text{ is adjacent to any vertex from } T\setminus V(C_x).
\end{gathered}
\end{align}
\noindent Otherwise, starting from $T$ we could shift the vertices around $C_x$ and would obtain vertices $v^* \in V(D)$, $x^* \in X_{v^*} \cap V(C_x)$ and an acyclic transversal $T^*$ of $(H,X)/v^*$ such that the neighbors of $x^*$ on $C_x$ are in $T^*$ and such that $x^*$ has another in- or out-neighbor in $T^*$, contradicting Proposition~\ref{prop_forbidden_config}(a). Finally, we conclude that
\begin{gather} \label{eq_Cx-novertex}
\text{no vertex from } \{x_3,x_4,\ldots,x_p\} \text{ is in }  V(C_x).
\end{gather}
\noindent Assume, to the contrary, that there is an index $i \neq 2$ with $x_i \in V(C_x)$. Then, as $C$ is induced and since $x_ix_{i+1}$ as well as $x_{i-1}x_i$ are not arcs of $H$, both neighbors of $x_i$ in $C_x$ must be from $V(H) \setminus \{x_2,x_3,\ldots,x_p\}$. But then, starting from $T$ we can shift $x_2 \to x, x_3 \to x_2, \ldots, x_i \to x_{i-1}$ and obtain an acyclic transversal $\tilde{T}$ of $(X,H)/v_i$ such that $x_i$ either has two in- or out-neighbors from $\tilde{T}$, contradicting Proposition~\ref{prop_forbidden_config}(a).

\noindent By analogous arguments we conclude that $H[T' \cup \{x\}]$ contains a directed cycle $C_x'$ and $x$ and $x_p'$ are consecutive on $C_x'$. Furthermore, if $z'$ denotes the vertex different from $x_p'$ such that $x$ and $z'$ are consecutive on $C_x'$, we have $z \not \in \{x_2',x_3',\ldots,x_{p-1}'\}$. Moreover, the following holds:
\begin{align}\label{eq_Cx'induced}
\begin{gathered}
C_x' \text{ is an induced directed cycle of  } H[T' \cup \{x\}] \text{ and}\\
\text{no vertex from  } C_x' \text{ is adjacent to any vertex from }T'\setminus V(C_x')
\end{gathered}
\end{align}
\noindent and
\begin{gather}\label{eq_Cx'-novertex}
\text{no vertex from  } \{x_2',x_3',\ldots,x_{p-1}'\} \text{ is in  }V(C_x').
\end{gather}
\noindent Since $T \setminus \{x_2,x_3,\ldots,x_p\} = T' \setminus \{x_2',x_3',\ldots,x_p'\}$, it follows from Proposition~\ref{prop_forbidden_config}(a) that $z=z'$. Let $y$ denote the vertex from $C_x$ different from $x$ such that $x_2$ and $y$ are consecutive on $C_x$ and let $y'$ denote the vertex from $C_x'$ different from $x$ such that $x_p'$ and $y'$ are consecutive on $C_x'$. Then, by combining \eqref{eq_Cx-induced}, \eqref{eq_Cx-novertex}, \eqref{eq_Cx'induced} and \eqref{eq_Cx'-novertex} with the fact that $T \setminus \{x_2,x_3,\ldots,x_p\}=T' \setminus \{x_2',x_3',\ldots,x_p'\}$, we conclude that $y=y'$ and that $H[V(C_x) \setminus \{x_2\}] = H[V(C_x') \setminus \{x_p'\}]$ is an induced directed path of $H$.Let $v \in V(D)$ denote the vertex such that $y \in X_v$. Then we have $v_2v \in A(D)$ and $v_pv \in A(D)$ and so $\{v_1,v_2,v_p,v\}$ either induces a $K_4^-$ in $G(D)$ (which is impossible by Claim~\ref{claim_K_^4-}) or a cycle $C'$ of length $4$ in $G(D)$ such that $D_{C'}$ is non-alternating in $D$, contradicting \eqref{eq_alternating}. This proves the claim. 
\end{proof2}

\begin{claim}\label{claim_allcyclesinduced}
All cycles in $G(D)$ are induced, i.e., no cycle has a chord.
\end{claim}

\begin{proof2}
Let $C$ be a cycle in $G(D)$. We prove that $C$ cannot contain a chord by induction on the length $p$ of $C$. If $p=4$, then $C$ has no chord as otherwise, the vertices of $C$ would either induce a $K_4$ or a $K_4^-$ in $G(D)$, contradicting Claim~\ref{claim_K_4} or Claim~\ref{claim_K_^4-}. Now assume $p \geq 5$. If $C$ has a chord, say $uv \in E(G)$, then the edge $uv$ divides the cycle $C$ into two smaller cycles $C_1$ and $C_2$. Then it follows from the induction hypothesis that neither $C_1$ nor $C_2$ has a chord. Hence, $C_1$ and $C_2$ are induced cycles of $G(D)$, and Claim~\ref{claim_induced-cycle->cyclic} implies that $D_{C_1}$ and $D_{C_2}$ are directed cycles. Furthermore, $uv$ is the only chord of $C$, since otherwise $G[V(C)]$ would contain a smaller cycle than $C$ whose edges would have no cyclic orientation in $D$,  contradicting Claim~\ref{claim_induced-cycle->cyclic}. By symmetry, we may assume that $uv \in A(D)$. Then, in $D_C$ the vertex $u$ has two in-neighbors, and the vertex $v$ has two out-neighbors, say $w$ and $w'$. Moreover, by symmetry, $C_1$ contains the vertices $u,v,$ and $w$ and $C_2$ contains the vertices $u,v,$ and $w'$. Let $T$ be an acyclic transversal of $(X,H)/v$ and let $u_1 \in X_u \cap T$, $w_1 \in X_w \cap T$, and $w_1' \in X_{w'} \cap T$. Furthermore we choose a cyclic ordering of the vertices of $C$ such that $w$ is the left neighbor of $v$ and $w'$ is the right neighbor. Then, there are vertices $v_1,v_2,v_3 \in X_v$ with $v_1w_1, v_2w_1'$ and $u_1v_3 \in A(H)$ (by Proposition~\ref{prop_forbidden_config}(b),(c)). Furthermore, by Proposition~\ref{prop_forbidden_config}(a), $v_1 \neq v_2$.
By shifting $w \to v$ and the remaining vertices of $C$ (except $v_1$) counterclockwise around $C$, we get an acyclic transversal $T'$ of $(X,H)/w'$ with $v_1 \in T'$. Thus, by Proposition~\ref{prop_forbidden_config}(c), there is a vertex $w_2' \in X_{w'}$ with $v_1w_2' \in A(H)$. In particular, $w_2' \neq w_1'$ (as $v_1 \neq v_2)$. By similar argumentation, $v_2$ has an out-neighbor $w_2 \neq w_1$ from $X_{w}$ (see Figure~\ref{fig_1-claim7}). Now we claim that $v_3 \not \in \{v_1,v_2\}$. Assume that $v_3=v_1$. Then, starting from $T$, we can shift each vertex from $C_2$ counterclockwise (beginning with $u \to v$) around $C_2$ (which gives us a transversal of $(X,H)/w'$ containing $v_1$) and, afterwards shift $v \to w'$. Then we get an acyclic transversal $T^*$ of $(X,H)/v$ that contains $w_1$ as well as $w_2'$ and so $|N_H^+(v_1) \cap T^*| \geq 2$, a contradiction to Proposition~\ref{prop_forbidden_config}(a). Hence, $v_3 \neq v_1$. By repeating the argumentation with $C_1$ instead of $C_2$ we conclude that $v_3 \neq v_2$. Clearly, $v_3$ has an out-neighbor $w_3' \in X_{w'}$ and an out-neighbor $w_3 \in X_w$ (shift clockwise around $C_2$, respectively $C_1$). This is also illustrated in Figure~\ref{fig_2-claim7}. By (C2) and since $v_3 \not \in \{v_1,v_2\}$, the vertex $w_3'$ is neither $w_1'$ nor $w_2'$. Now finally, starting from $T$, we shift each vertex (beginning with $u \to v$, i.e. $u_1 \to v_3$) counterclockwise around $C_2$ such that we get an acyclic transversal of $(X,H)/w'$ and, afterwards, we shift $v \to w'$ (i.e. $v_3 \to w_3'$). This gives us an acyclic transversal $\tilde{T}$ of $(X,H)/v$ with $w_3' \in \tilde{T}$. We claim that $v_2$ has no out-neighbor in $\tilde{T}$ (which would contradict Proposition~\ref{prop_forbidden_config}(a)). As $uv$ is the unique chord of $C$, we conclude that $w \not \in V(C_2)$ and so $w_1 \in \tilde{T}$. Since $v_1w_1 \in A(H)$, (C2) implies that $v_2w_1 \not \in  A(H)$. Furthermore, the out-neighbor of $v_2$ from $\tilde{T}$ must be contained in $\bigcup_{v' \in V(C_2)} X_{v'}$ as $w_1'$ is the out-neighbor of $v_2$ from $T$ and since we only shifted around $C_2$. But since $C_2$ has no chords and since $vu \not \in A(H)$, the out-neighbor of $v_2$ from $\tilde{T}$ can only be the vertex from $X_{w'} \cap \tilde{T}$, that is, $w_3'$. However, $v_3w_3' \in A(H)$ and so $v_2w_3' \not \in A(H)$. Thus, $v_2$ has not out-neighbor from $\tilde{T}$, a contradiction. This proves the claim.
\end{proof2}

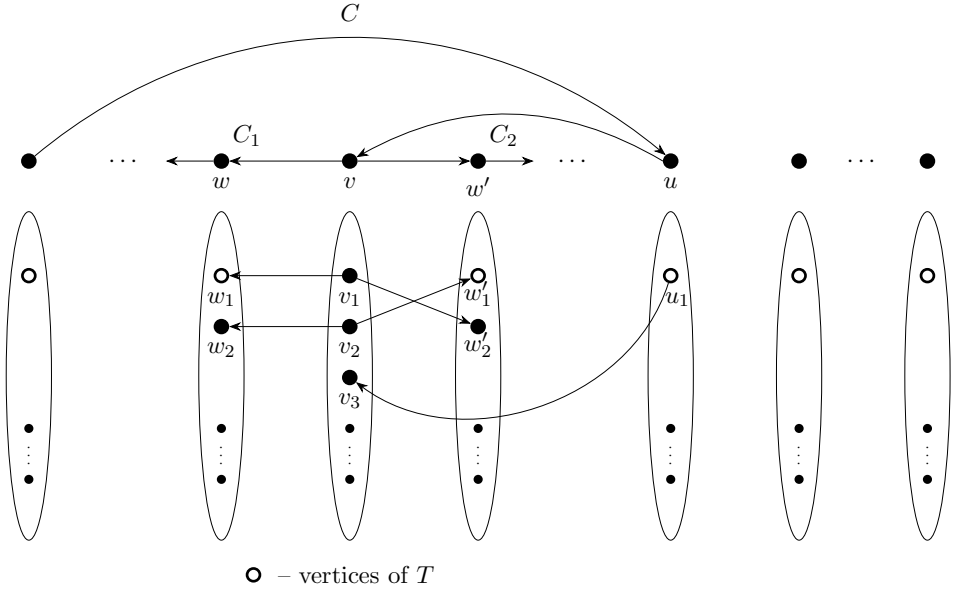
\begin{figure}[h]
\centering
\resizebox{.85\linewidth}{!}{
\input{claim7-fig1}
}
\caption{Setting up $(D,X,H)$.}
\label{fig_1-claim7}
\end{figure}

\begin{figure}[h]
\centering
\resizebox{.85\linewidth}{!}{
\input{claim7-fig2}
}
\caption{Including the neighbors of $v_3$.}
\label{fig_2-claim7}
\end{figure}
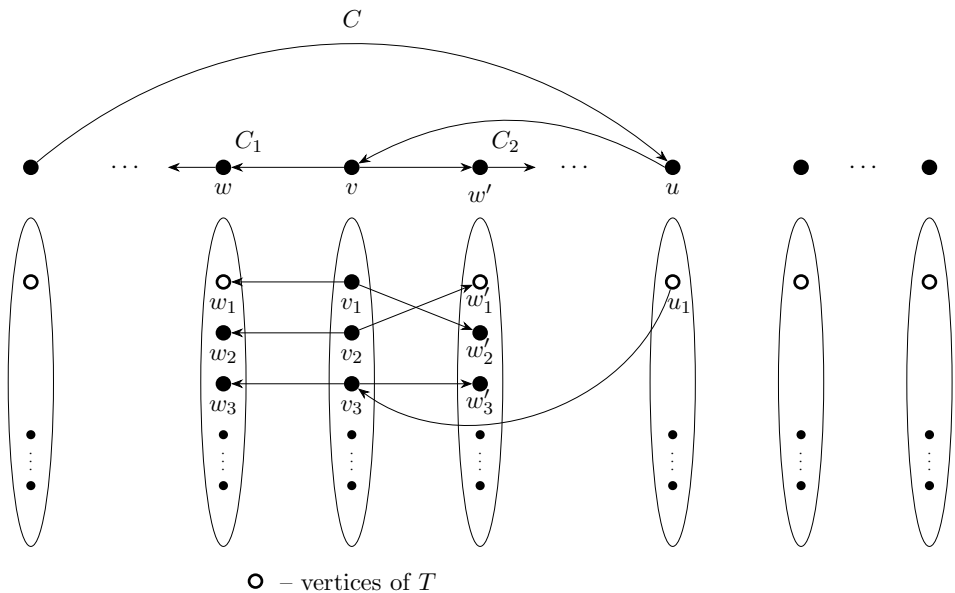

The remaining part of the proof is straightforward: As $D$ is a block, $G(D)$ contains an induced cycle $C$. Then, $D_C$ is a directed cycle by Claim~\ref{claim_induced-cycle->cyclic}. We claim that $D=D_C$. Otherwise, there would be a vertex $v \in V(D) \setminus V(C)$. Moreover, since $D$ and therefore $G(D)$ is a block, there are two internally disjoint paths $P$ and $P'$ in $G(D)$ from $v$ to vertices $w \neq w'$ such that $V(P) \cap V(C) = \{w\}$ and $V(P') \cap V(C) = \{w'\}$. Since all cycles of $G(D)$ are induced (by Claim~\ref{claim_allcyclesinduced}), $w$ and $w'$ are not consecutive in $C$. Let $P_C$ and $P_C'$ denote the two internally disjoint paths between $w$ and $w'$ contained in $C$. Then, $P,P'$ together with $P_C$, respectively $P,P'$ together with $P_C'$ form induced cycles $C_1$ and $C_2$ of $G(D)$. Since $D_C$ is a directed cycle, either $D_{C_1}$ or $D_{C_2}$ is not a directed cycle, contradicting Claim~\ref{claim_induced-cycle->cyclic}. Hence, $D=D_C$, \emph{i.e.}, $D$ is a directed cycle. As $(D,X,H)$ is a minimal uncolorable degree-feasible configuration, we easily conclude that $(D,X,H)$ is a C-configuration. This completes the proof. 
\end{proof}

\section{Concluding Remarks}

The next two statements are direct consequences of Theorem~\ref{theorem:main-result} and Proposition~\ref{prop_block-structure}. In particular, Theorem~\ref{theorem_harut-extended} is a generalization of Theorem~\ref{theorem_harut}.

\begin{corollary}
Let $(D,X,H)$ be a degree-feasible configuration. If $(D,X,H)$ is minimal uncolorable, then for each block $B \in \mathcal{B}(D)$ there is a uniquely determined cover $(X^B,H^B)$ of $B$ such that the following statements hold.
\begin{itemize}
\item[\upshape (a)] For every block $B \in \mathcal{B}(D)$, the triple $(B,X^B,H^B)$ is a {\upshape K}-configuration, a {\upshape C}-configuration, or a {\upshape BC}-configuration.
\item[\upshape (b)] The digraphs $H^B$ with $B \in \mathcal{B}(D)$ are pairwise disjoint and $H= \bigcup_{B \in \mathcal{B}(D)} H^B$.
\item[\upshape (c)] For each vertex $v \in V(D)$ it holds $X_v=\bigcup_{B \in \mathcal{B}(D), v \in V(B)}X_v^B$.
\end{itemize}
\end{corollary}

\begin{theorem}\label{theorem_harut-extended}
A connected digraph $D$ is not  {\upshape DP}-degree-colorable if and only if for every block $B$ of $D$ one of the following cases occurs:
\begin{itemize}
\item[\upshape (a)] $B$ is a directed cycle of length $\geq 2$.
\item[\upshape (b)] $B$ is a bidirected cycle of length $\geq 3$.
\item[\upshape (c)] $B$ is a bidirected complete graph.
\end{itemize}
\end{theorem}

Finally, we deduce a Brooks-type theorem for DP-colorings of digraphs. For undirected graphs, the theorem was proven by Bernshteyn, Kostochka, and Pron \cite{BeKoPro17}.

\begin{theorem}
Let $D$ be a connected digraph. Then, $\chi_{\rm{DP}}(D) \leq \max \{\Delta^+(D), \Delta^-(D)\} + 1$ and equality holds if and only if $D$ is  
\begin{itemize}
\item[\upshape (a)] a directed cycle of length $\geq 2$, or
\item[\upshape (b)] a bidirected cycle of length $\geq 3$, or
\item[\upshape (c)] a bidirected complete graph.
\end{itemize}
\end{theorem}
\begin{proof}
As mentioned earlier, $\chi_{\text{DP}}(D) \leq \max \{\Delta^+(D), \Delta^-(D)\} + 1$ is always true. Moreover, if $D$ satisfies (a),(b), or (c), then $\chi_{\text{DP}}(D) = \max\{\Delta(D)^+, \Delta^-(D)\} + 1$, just take a C-, BC-, or K-configuration. Now assume $\chi_{\text{DP}}(D) = \max \{\Delta^+(D), \Delta^-(D)\} + 1$. Then, there is a cover $(X,H)$ of $D$ such that $|X_v| \geq \max \{\Delta^+(D), \Delta^-(D)\}$ for all $v \in V(D)$ and $D$ is not $(X,H)$-colorable. Hence, $(D,X,H)$ is an uncolorable degree-feasible configuration and there is a spanning subdigraph $H'$ of $H$ such that $(D,X,H')$ is minimal uncolorable. Then, $|X_v|=d_D^+(v)=d_D^-(v)$ for all $v \in V(G)$ (by Proposition~\ref{prop_main-prop}(a)) and each block of $D$ satisfies (a),(b) or (c) (by Theorem~\ref{theorem_harut-extended}). Thus, $|X_v| = \max \{\Delta^+(D), \Delta^-(D)\}$ for all $v \in V(D)$ and we conclude that $D$ has only one block and, therefore, satisfies (a), (b) or (c). This completes the proof.
\end{proof}

In 1996, Johansson~\cite{Jo96} proved that $\chi(G)=\mathcal{O}( \frac{\Delta(G)}{\log_2\Delta(G)})$ provided that the undirected graph $G$ contains no triangle.  Regarding digraphs, Erd\H{o}s~\cite{Erd79} conjectured that $\chi(D)=\mathcal{O}(\frac{\Delta(D)}{\log_2\Delta(D)})$ for digon-free digraphs, whereas $\Delta(D)$ denotes the maximum total degree of $D$. To the knowledge of the authors, this conjecture is still open.  Related to this question, Harutyunyan and Mohar \cite{HaMo11.2}  proved the following. Given a digraph $D$, let $\tilde{\Delta}(D)=\max\{\sqrt{d^+(v)d^-(v)} ~|~ v \in V(D)\}$.

\begin{theorem}[Harutyunyan and Mohar]
There is an absolute constant $\Delta_1$ such that every digon-free digraph $D$ with $\tilde{\Delta}(D) \geq \Delta_1$ has $\chi(D)  \leq (1- e^{-13})\tilde{\Delta}(D)$.
\end{theorem}

Moreover, Bensmail, Harutyunyan and Khang Le \cite{BeHaKha15} managed to extend the above theorem to list-colorings of digon-free digraphs.

\begin{theorem}[Bensmail, Harutyunyan and Khang Le]
There is an absolute constant $\Delta_1$ such that every digon-free digraph $D$ with $\tilde{\Delta}(D) \geq \Delta_1$ has $\chi_{\ell}(D)  \leq (1- e^{-18})\tilde{\Delta}(D)$.
\end{theorem}

Thus, it is a natural question to ask whether this theorem can be transferred to DP-colorings of digon-free digraphs and the authors encourage the reader to try his luck. 

Another problem that may be worth examining is the following. In \cite{Oh02}, Ohba conjectured that for graphs with few vertices compared to their chromatic number the chromatic number and the list-chromatic number coincide. This conjecture was recently proven by Noel, Reed, and Wu in \cite{NoReWu14}.

\begin{theorem}[Ohba's Conjecture]
For every graph $G$ satisfying $\chi(G) \geq (|G|-1)/2$, we have $\chi(G)=\chi_{\ell}(G)$.
\end{theorem}

In \cite{BeHaKha15}, a simple transformation is used in order to obtain the directed version of Ohba's Conjecture from the undirected case.

\begin{theorem}
For every digraph $D$ satisfying $\chi(D) \geq (|D|-1)/2$, we have $\chi(D)=\chi_\ell(D)$.
\end{theorem}

It is easy to see that Ohba's Conjecture does not hold if we take DP-colorings instead of list-colorings neither in the undirected nor in the directed case (just take a $C_4$, or a bidirected $C_4$, respectively). However, Bernshteyn, Kostochka and Zhu~\cite{BeKoZhu17} proved the following, sharp, bound. 

\begin{theorem} For $n \in \mathbb{N}$, let $r(n)$ denote the minimum $r \in \mathbb{N}$ such that for every $n$-vertex graph $G$ with $\chi(G) \geq r$, we have $\chi_{DP}(G)=\chi(G)$. Then, $$n - r(n) = \Theta(\sqrt{n}).$$
\end{theorem}

It seems very likely that it is possible to transfer the above theorem to DP-colorings of directed graphs.

\end{document}

%% file: k_config_dir.tex
\begin{tikzpicture}

\node[vertex, label={below:$1$}] (u1){};
\node[vertex, label={below:$2$}, xshift=2cm] (u2) at (u1) {};
\node[vertex, label={below:$3$},  xshift=2cm] (u3) at (u2) {};
\node[vertex,label={below:$4$},  xshift=2cm] (u4) at (u3) {};

\path[arrows={[scale=1.1]Stealth}-{[scale=1.1]Stealth}]
		(u1)
		edge (u2)
		edge [bend left] (u3)
		edge[bend left] (u4)	
		(u2)
		edge[] (u3)
		edge[ bend left] (u4)
		
		(u3)
		edge(u4);

\begin{scope}[bend angle=10]
\node[vertex, yshift=-1.8cm] (u11) at (u1) {};
\node[vertex, yshift=-.8cm] (u12) at (u11){};
\node[vertex, yshift=-.8cm] (u13) at (u12){};

\node[vertex, yshift=-1.8cm] (u21) at (u2) {};
\node[vertex, yshift=-.8cm] (u22) at (u21){};
\node[vertex, yshift=-.8cm] (u23) at (u22){};

\node[vertex, yshift=-1.8cm] (u31) at (u3) {};
\node[vertex, yshift=-.8cm] (u32) at (u31){};
\node[vertex, yshift=-.8cm] (u33) at (u32){};

\node[vertex, yshift=-1.8cm] (u41) at (u4) {};
\node[vertex, yshift=-.8cm] (u42) at (u41){};
\node[vertex, yshift=-.8cm] (u43) at (u42){};
\end{scope}

\path[arrows={[scale=1.1]Stealth}-{[scale=1.1]Stealth}]
(u11) edge (u21)
	  edge[ bend left] (u31)
	  edge[ bend left] (u41)
(u12) edge (u22)
	  edge[ bend left] (u32)
	  edge[ bend left] (u42)

(u13) edge (u23)
	  edge[ bend left] (u33)
	  edge[ bend left] (u43)
	  
(u21) edge (u31)
	  edge[ bend left] (u41)

(u22) edge (u32)
	  edge[ bend left] (u42)
	  
(u23) edge (u33)
		   edge[bend left] (u43)	  

(u31) edge (u41)

(u32) edge (u42)
(u33) edge (u43);

\begin{pgfonlayer}{background}
\node[ellipse, draw=black, fit=(u11)(u12)(u13), label={below:$X_1$}]{};
\node[ellipse, draw=black, fit=(u21)(u22)(u23), label={below:$X_2$}]{};
\node[ellipse, draw=black, fit=(u31)(u32)(u33), label={below:$X_3$}]{};
\node[ellipse, draw=black, fit=(u41)(u42)(u43), label={below:$X_4$}]{};
\end{pgfonlayer}

\node (h1) at (u4) [xshift=1cm, yshift=1cm]{};
\node (h2) at (u43)[xshift=1cm, yshift=-1.3cm]{};

\draw[dashed] (h1) -- (h2);

\begin{scope}[xshift=8cm]
\node[vertex, label={below:$1$}] (u1){};
\node[vertex, label={below:$2$}, xshift=2cm] (u2) at (u1) {};
\node[vertex, label={below:$3$},  xshift=2cm] (u3) at (u2) {};
\node[vertex,label={below:$4$},  xshift=2cm] (u4) at (u3) {};

\path[arrows={[scale=1.1]Stealth}-{[scale=1.1]Stealth}]
		(u1)
		edge (u2)
		edge[bend left](u4)
		(u2)
		edge (u3)
		
		(u3)
		edge(u4);

\begin{scope}[bend angle=10]
\node[vertex, yshift=-2.2cm] (u11) at (u1) {};
\node[vertex, yshift=-.8cm] (u12) at (u11){};

\node[vertex, yshift=-2.2cm] (u21) at (u2) {};
\node[vertex, yshift=-.8cm] (u22) at (u21){};

\node[vertex, yshift=-2.2cm] (u31) at (u3) {};
\node[vertex, yshift=-.8cm] (u32) at (u31){};

\node[vertex, yshift=-2.2cm] (u41) at (u4) {};
\node[vertex, yshift=-.8cm] (u42) at (u41){};
\end{scope}

\path[arrows={[scale=1.1]Stealth}-{[scale=1.1]Stealth}]
(u11) edge (u21)
	  edge (u42)
(u12) edge (u22)
	  edge(u41)
(u21) edge (u31)
(u22) edge (u32)
(u31) edge (u41)
(u32) edge (u42);

\begin{pgfonlayer}{background}
\node[ellipse, draw=black, fit=(u11)(u12), label={below:$X_1$}]{};
\node[ellipse, draw=black, fit=(u21)(u22), label={below:$X_2$}]{};
\node[ellipse, draw=black, fit=(u31)(u32), label={below:$X_3$}]{};
\node[ellipse, draw=black, fit=(u41)(u42), label={below:$X_4$}]{};
\end{pgfonlayer} 

\end{scope}
\end{tikzpicture}

%% file: k_c_config.tex
\begin{tikzpicture}

\node[vertex, label={below:$1$}] (u1){};
\node[vertex, label={below:$2$}, xshift=2cm] (u2) at (u1) {};
\node[vertex, label={below:$3$},  xshift=2cm] (u3) at (u2) {};
\node[vertex,label={below:$4$},  xshift=2cm] (u4) at (u3) {};

\path[-]
		(u1)
		edge[help lines] (u2)
		edge[help lines, bend left] (u3)
		edge[help lines, bend left] (u4)	
		(u2)
		edge[help lines] (u3)
		edge[help lines, bend left] (u4)
		
		(u3)
		edge[help lines](u4);

\begin{scope}[bend angle=10]
\node[vertex, yshift=-1.8cm] (u11) at (u1) {};
\node[vertex, yshift=-.8cm] (u12) at (u11){};
\node[vertex, yshift=-.8cm] (u13) at (u12){};

\node[vertex, yshift=-1.8cm] (u21) at (u2) {};
\node[vertex, yshift=-.8cm] (u22) at (u21){};
\node[vertex, yshift=-.8cm] (u23) at (u22){};

\node[vertex, yshift=-1.8cm] (u31) at (u3) {};
\node[vertex, yshift=-.8cm] (u32) at (u31){};
\node[vertex, yshift=-.8cm] (u33) at (u32){};

\node[vertex, yshift=-1.8cm] (u41) at (u4) {};
\node[vertex, yshift=-.8cm] (u42) at (u41){};
\node[vertex, yshift=-.8cm] (u43) at (u42){};
\end{scope}

\path[-]
(u11) edge (u21)
	  edge[ bend left] (u31)
	  edge[ bend left] (u41)
(u12) edge (u22)
	  edge[ bend left] (u32)
	  edge[ bend left] (u42)

(u13) edge (u23)
	  edge[ bend left] (u33)
	  edge[ bend left] (u43)
	  
(u21) edge (u31)
	  edge[ bend left] (u41)

(u22) edge (u32)
	  edge[ bend left] (u42)
	  
(u23) edge (u33)
		   edge[bend left] (u43)	  

(u31) edge (u41)

(u32) edge (u42)
(u33) edge (u43);

\begin{pgfonlayer}{background}
\node[ellipse, draw=black, fit=(u11)(u12)(u13), label={below:$X_1$}]{};
\node[ellipse, draw=black, fit=(u21)(u22)(u23), label={below:$X_2$}]{};
\node[ellipse, draw=black, fit=(u31)(u32)(u33), label={below:$X_3$}]{};
\node[ellipse, draw=black, fit=(u41)(u42)(u43), label={below:$X_4$}]{};
\end{pgfonlayer}

\node (h1) at (u4) [xshift=1cm, yshift=1cm]{};
\node (h2) at (u43)[xshift=1cm, yshift=-1.3cm]{};

\draw[dashed] (h1) -- (h2);

\begin{scope}[xshift=8cm]
\node[vertex, label={below:$1$}] (u1){};
\node[vertex, label={below:$2$}, xshift=2cm] (u2) at (u1) {};
\node[vertex, label={below:$3$},  xshift=2cm] (u3) at (u2) {};
\node[vertex,label={below:$4$},  xshift=2cm] (u4) at (u3) {};
\node[vertex, label={below:$5$}, xshift=2cm] (u5) at (u4) {};

\path[-]
		(u1)
		edge[help lines] (u2)
		edge[help lines, bend left] (u5)

		(u2)
		edge[help lines] (u3)
		
		(u3)
		edge[help lines](u4)
		
		(u4)
		edge[help lines](u5);

\begin{scope}[bend angle=10]
\node[vertex, yshift=-2.2cm] (u11) at (u1) {};
\node[vertex, yshift=-.8cm] (u12) at (u11){};

\node[vertex, yshift=-2.2cm] (u21) at (u2) {};
\node[vertex, yshift=-.8cm] (u22) at (u21){};

\node[vertex, yshift=-2.2cm] (u31) at (u3) {};
\node[vertex, yshift=-.8cm] (u32) at (u31){};

\node[vertex, yshift=-2.2cm] (u41) at (u4) {};
\node[vertex, yshift=-.8cm] (u42) at (u41){};

\node[vertex, yshift=-2.2cm] (u51) at (u5) {};
\node[vertex, yshift=-.8cm] (u52) at (u51) {};
\end{scope}

\path[-]
(u11) edge[help lines] (u21)
	  edge[help lines, bend left] (u51)
(u12) edge[help lines] (u22)
	  edge[help lines, bend right](u52)
(u21) edge[help lines] (u31)
(u22) edge[help lines] (u32)
(u31) edge[help lines] (u41)
(u32) edge[help lines] (u42)
(u41) edge[help lines] (u51)
(u42) edge[help lines] (u52); 

\begin{pgfonlayer}{background}
\node[ellipse, draw=black, fit=(u11)(u12), label={below:$X_1$}]{};
\node[ellipse, draw=black, fit=(u21)(u22), label={below:$X_2$}]{};
\node[ellipse, draw=black, fit=(u31)(u32), label={below:$X_3$}]{};
\node[ellipse, draw=black, fit=(u41)(u42), label={below:$X_4$}]{};
\node[ellipse, draw=black, fit=(u51)(u52), label={below:$X_5$}]{};
\end{pgfonlayer} 

\end{scope}
\end{tikzpicture}

%% file: forbidden_config.tex
\begin{tikzpicture}[>={[scale=1.1]Stealth},
tinyvertex/.style={circle,minimum size=1mm,very thick, draw=black, fill=black, inner sep=0mm}]

\node[vertex, label={below:$v$}] (u1){};
\node[vertex, xshift=2cm] (u2) at (u1) {};
\node[vertex, xshift=2cm] (u3) at (u2) {};

\path[-]
		(u1)
		edge [->] (u2)
		edge [->,bend left] (u3);

\begin{scope}[bend angle=10]
\node[vertex, label={[font=\tiny] below:$x$}, yshift=-1.8cm] (u11) at (u1) {};
\node[tinyvertex, yshift=-.8cm] (u12) at (u11){};
\node[yshift=-.4cm,  font=\tiny,rotate=90] at (u12) {$\cdots$};
\node[tinyvertex, yshift=-.8cm] (u13) at (u12){};

\node[vertex, fill=white, yshift=-1.8cm] (u21) at (u2) {};
\node[tinyvertex, yshift=-.8cm] (u23) at (u21){};
\node[yshift=-.4cm,  font=\tiny,rotate=90] at (u23) {$\cdots$};
\node[tinyvertex, yshift=-.8cm] (u24) at (u23){};

\node[vertex, fill=white, yshift=-1.8cm] (u31) at (u3) {};
\node[tinyvertex, yshift=-.8cm] (u33) at (u31){};
\node[yshift=-.4cm,  font=\tiny,rotate=90] at (u33) {$\cdots$};
\node[tinyvertex, yshift=-.8cm] (u34) at (u33){};

\path[-]
(u11) edge[->, bend left] (u31)
	  edge[->] (u21);

\begin{pgfonlayer}{background}
\node[ellipse, draw=black, fit=(u21)(u24), xshift=-2cm]{};
\node[ellipse, draw=black, fit=(u21)(u24)]{};
\node[ellipse, draw=black, fit=(u31)(u34)]{};
\end{pgfonlayer}
\end{scope}

\node (h1) at (u3) [xshift=1cm, yshift=.5cm] {};
\node (h2) at (u34)[xshift=1cm, yshift=-1cm]{};
\draw[dashed] (h1)--(h2);


\begin{scope}[xshift=6cm]

\node[vertex, label={below:$v$}] (u1){};
\node[vertex,label={below:$u$}, xshift=2cm] (u2) at (u1) {};
\node[vertex, xshift=2cm] (u3) at (u2) {};

\path[-]
		(u1)
		edge [->] (u2)
		(u2) edge [<-] (u3);

\begin{scope}[bend angle=10]
\node[vertex, label={[font=\tiny] below:$x$}, yshift=-1.8cm] (u11) at (u1) {};
\node[tinyvertex, yshift=-.8cm] (u12) at (u11){};
\node[yshift=-.4cm,  font=\tiny,rotate=90] at (u12) {$\cdots$};
\node[tinyvertex, yshift=-.8cm] (u13) at (u12){};

\node[vertex, fill=white, label={[font=\tiny] below:$x_u$}, yshift=-1.8cm] (u21) at (u2) {};
\node[tinyvertex, yshift=-.8cm] (u23) at (u21){};
\node[yshift=-.4cm,  font=\tiny,rotate=90] at (u23) {$\cdots$};
\node[tinyvertex, yshift=-.8cm] (u24) at (u23){};

\node[vertex, fill=white, yshift=-1.8cm] (u31) at (u3) {};
\node[tinyvertex, yshift=-.8cm] (u33) at (u31){};
\node[yshift=-.4cm,  font=\tiny,rotate=90] at (u33) {$\cdots$};
\node[tinyvertex, yshift=-.8cm] (u34) at (u33){};

\node[vertex, fill=white, yshift=-1.5cm, xshift=.5cm] (wv) at (u13) {};
\node [xshift=1.6cm] at (wv) {-- vertices of $T$};

\path[-]
(u11) edge[->] (u21)
(u31) edge[->] (u21);

\begin{pgfonlayer}{background}
\node[ellipse, draw=black, fit=(u21)(u24), xshift=-2cm]{};
\node[ellipse, draw=black, fit=(u21)(u24)]{};
\node[ellipse, draw=black, fit=(u31)(u34)]{};
\end{pgfonlayer}
\end{scope}

\node (h1) at (u3) [xshift=1cm, yshift=.5cm] {};
\node (h2) at (u34)[xshift=1cm, yshift=-1cm]{};
\draw[dashed] (h1)--(h2);
\end{scope}


\begin{scope}[xshift=12cm]

\node[vertex, label={below:$v$}] (u1){};
\node[vertex,label={below:$w$}, xshift=2cm] (u2) at (u1) {};
\node[vertex, xshift=2cm] (u3) at (u2) {};

\path[-]
		(u1)
		edge [<-] (u2)
		(u2) edge [->] (u3);

\begin{scope}[bend angle=10]
\node[vertex, label={[font=\tiny] below:$x$}, yshift=-1.8cm] (u11) at (u1) {};
\node[tinyvertex, yshift=-.8cm] (u12) at (u11){};
\node[yshift=-.4cm,  font=\tiny,rotate=90] at (u12) {$\cdots$};
\node[tinyvertex, yshift=-.8cm] (u13) at (u12){};

\node[vertex, fill=white, label={[font=\tiny] below:$x_w$},yshift=-1.8cm] (u21) at (u2) {};
\node[tinyvertex, yshift=-.8cm] (u23) at (u21){};
\node[yshift=-.4cm,  font=\tiny,rotate=90] at (u23) {$\cdots$};
\node[tinyvertex, yshift=-.8cm] (u24) at (u23){};

\node[vertex, fill=white, yshift=-1.8cm] (u31) at (u3) {};
\node[tinyvertex, yshift=-.8cm] (u33) at (u31){};
\node[yshift=-.4cm,  font=\tiny,rotate=90] at (u33) {$\cdots$};
\node[tinyvertex, yshift=-.8cm] (u34) at (u33){};

\path[-]
(u11) edge[<-] (u21)
(u31) edge[<-] (u21);

\begin{pgfonlayer}{background}
\node[ellipse, draw=black, fit=(u21)(u24), xshift=-2cm]{};
\node[ellipse, draw=black, fit=(u21)(u24)]{};
\node[ellipse, draw=black, fit=(u31)(u34)]{};
\end{pgfonlayer}
\end{scope}
\end{scope}

\end{tikzpicture}

%% file: claim1-fig1.tex
\begin{tikzpicture}[>={[scale=1.1]Stealth},
tinyvertex/.style={circle,minimum size=1mm,very thick, draw=black, fill=black, inner sep=0mm}]

\node[vertex, label={below:$v_1$}] (u1){};
\node[vertex, label={below:$v_2$}, xshift=2cm] (u2) at (u1) {};
\node[vertex, label={below:$v_3$},  xshift=2cm] (u3) at (u2) {};
\node[vertex,  xshift=2cm]  (u4) at (u3) {};
\node[xshift=1cm] at (u4) {$\cdots$};
\node[vertex, xshift=2cm] (u5) at (u4) {};

\path[-]
		(u1)
		edge [->] (u2)
		edge [<-,bend left] (u3)
		edge [->, bend left] (u4)
		(u2)
		edge[<-] (u3);
%

\begin{scope}[bend angle=10]
\node[vertex, label={ below:$x_1$}, yshift=-1.8cm] (u11) at (u1) {};
\node[tinyvertex, yshift=-1.8cm] (u12) at (u11){};
\node[yshift=-.4cm,  font=\tiny,rotate=90] at (u12) {$\cdots$};
\node[tinyvertex, yshift=-.8cm] (u13) at (u12){};

\node[vertex, fill=white, label={ below:$x_2$}, yshift=-1.8cm] (u21) at (u2) {};
\node[vertex, label={ below:$x_2'$}, yshift=-.8cm] (u22) at (u21){};
\node[tinyvertex, yshift=-1cm] (u23) at (u22){};
\node[yshift=-.4cm,  font=\tiny,rotate=90] at (u23) {$\cdots$};
\node[tinyvertex, yshift=-.8cm] (u24) at (u23){};

\node[vertex, fill=white, label={ below:$x_3$}, yshift=-1.8cm] (u31) at (u3) {};
\node[vertex, yshift=-.8cm] (u32) at (u31){};
\node[tinyvertex, yshift=-1cm] (u33) at (u32){};
\node[yshift=-.4cm,  font=\tiny,rotate=90] at (u33) {$\cdots$};
\node[tinyvertex, yshift=-.8cm] (u34) at (u33){};

\node[vertex, fill=white, label={ below:$x$}, yshift=-1.8cm] (u41) at (u4) {};
\node[tinyvertex, yshift=-1.8cm] (u42) at (u41){};
\node[yshift=-.4cm,  font=\tiny,rotate=90] at (u42) {$\cdots$};
\node[tinyvertex, yshift=-.8cm] (u43) at (u42){};

\node[vertex, fill=white, yshift=-1.8cm] (u51) at (u5) {};
\node[tinyvertex, yshift=-1.8cm] (u52) at (u51){};
\node[yshift=-.4cm,  font=\tiny,rotate=90] at (u52) {$\cdots$};
\node[tinyvertex, yshift=-.8cm] (u53) at (u52){};

\end{scope}	

\node[vertex, fill=white, yshift=-1.5cm, xshift=.5cm] (wv) at (u24) {};
\node [xshift=1.6cm] at (wv) {-- vertices of $T$};

\path[-]
(u11) edge[<-, bend left] (u31)
	  edge[->] (u22)
	  edge[->,bend left] (u41)

(u21) [<-] edge (u32);

\begin{pgfonlayer}{background}
\node[ellipse, draw=black, fit=(u21)(u22)(u23)(u24), xshift=-2cm]{};
\node[ellipse, draw=black, fit=(u21)(u22)(u23)(u24)]{};
\node[ellipse, draw=black, fit=(u31)(u32)(u34)]{};
\node[ellipse, draw=black, fit=(u21)(u22)(u23)(u24), xshift=4cm]{};
\node[ellipse, draw=black, fit=(u21)(u22)(u23)(u24), xshift=6cm]{};
\end{pgfonlayer}

\node (h1) at (u5) [xshift=1cm, yshift=1cm]{};
\node (h3) at (u13)[xshift=-.5cm, yshift=-2.5cm]{};

%
%
\begin{scope}[xshift=10cm]
\node[vertex, label={below:$v_1$}] (u1){};
\node[vertex, label={below:$v_2$}, xshift=2cm] (u2) at (u1) {};
\node[vertex, label={below:$v_3$},  xshift=2cm] (u3) at (u2) {};
\node[vertex,  xshift=2cm]  (u4) at (u3) {};
\node[xshift=1cm] at (u4) {$\cdots$};
\node[vertex, xshift=2cm] (u5) at (u4) {};

\path[-]
		(u1)
		edge [->] (u2)
		edge [<-,bend left] (u3)
		edge [->, bend left] (u4)
		(u2)
		edge[<-] (u3);
%

\begin{scope}[bend angle=10]
\node[vertex, fill=white, label={ below:$x_1$}, yshift=-1.8cm] (u11) at (u1) {};
\node[tinyvertex, yshift=-1.8cm] (u12) at (u11){};
\node[yshift=-.4cm,  font=\tiny,rotate=90] at (u12) {$\cdots$};
\node[tinyvertex, yshift=-.8cm] (u13) at (u12){};

\node[vertex, fill=white, label={ below:$x_2$}, yshift=-1.8cm] (u21) at (u2) {};
\node[vertex, label={ below:$x_2'$}, yshift=-.8cm] (u22) at (u21){};
\node[tinyvertex, yshift=-1cm] (u23) at (u22){};
\node[yshift=-.4cm,  font=\tiny,rotate=90] at (u23) {$\cdots$};
\node[tinyvertex, yshift=-.8cm] (u24) at (u23){};

\node[vertex, label={ below:$x_3$}, yshift=-1.8cm] (u31) at (u3) {};
\node[vertex, label={ below:$x_3'$}, yshift=-.8cm] (u32) at (u31){};
\node[tinyvertex, yshift=-1cm] (u33) at (u32){};
\node[yshift=-.4cm,  font=\tiny,rotate=90] at (u33) {$\cdots$};
\node[tinyvertex, yshift=-.8cm] (u34) at (u33){};

\node[vertex, fill=white, label={ below:$x$}, yshift=-1.8cm] (u41) at (u4) {};
\node[tinyvertex, yshift=-1.8cm] (u42) at (u41){};
\node[yshift=-.4cm,  font=\tiny,rotate=90] at (u42) {$\cdots$};
\node[tinyvertex, yshift=-.8cm] (u43) at (u42){};

\node[vertex, fill=white, yshift=-1.8cm] (u51) at (u5) {};
\node[tinyvertex, yshift=-1.8cm] (u52) at (u51){};
\node[yshift=-.4cm,  font=\tiny,rotate=90] at (u52) {$\cdots$};
\node[tinyvertex, yshift=-.8cm] (u53) at (u52){};
\node [yshift=-1.5cm] at (u34) {$v_3 \to v_1$};
\node (h4) at (u53)[xshift=.5cm, yshift=-2.5cm]{};
\draw [dashed] (h3) -- (h4);

\end{scope}

\path[-]
(u11) edge[<-, bend left] (u31)
	  edge[->] (u22)
	  edge[->,bend left] (u41)

(u21) [<-] edge (u32);

\begin{pgfonlayer}{background}
\node[ellipse, draw=black, fit=(u21)(u22)(u23)(u24), xshift=-2cm]{};
\node[ellipse, draw=black, fit=(u21)(u22)(u23)(u24)]{};
\node[ellipse, draw=black, fit=(u31)(u32)(u34)]{};
\node[ellipse, draw=black, fit=(u21)(u22)(u23)(u24), xshift=4cm]{};
\node[ellipse, draw=black, fit=(u21)(u22)(u23)(u24), xshift=6cm]{};
\end{pgfonlayer}

\end{scope}

\begin{scope}[yshift=-8cm]
\node[vertex, label={below:$v_1$}] (u1){};
\node[vertex, label={below:$v_2$}, xshift=2cm] (u2) at (u1) {};
\node[vertex, label={below:$v_3$},  xshift=2cm] (u3) at (u2) {};
\node[vertex,  xshift=2cm]  (u4) at (u3) {};
\node[xshift=1cm] at (u4) {$\cdots$};
\node[vertex, xshift=2cm] (u5) at (u4) {};

\path[-]
		(u1)
		edge [->] (u2)
		edge [<-,bend left] (u3)
		edge [->, bend left] (u4)
		(u2)
		edge[<-] (u3);
%

\begin{scope}[bend angle=10]
\node[vertex, fill=white, label={ below:$x_1$}, yshift=-1.8cm] (u11) at (u1) {};
\node[tinyvertex, yshift=-1.8cm] (u12) at (u11){};
\node[yshift=-.4cm,  font=\tiny,rotate=90] at (u12) {$\cdots$};
\node[tinyvertex, yshift=-.8cm] (u13) at (u12){};

\node[vertex, label={ below:$x_2$}, yshift=-1.8cm] (u21) at (u2) {};
\node[vertex, label={ below:$x_2'$}, yshift=-.8cm] (u22) at (u21){};
\node[tinyvertex, yshift=-1cm] (u23) at (u22){};
\node[yshift=-.4cm,  font=\tiny,rotate=90] at (u23) {$\cdots$};
\node[tinyvertex, yshift=-.8cm] (u24) at (u23){};

\node[vertex, label={ below:$x_3$}, yshift=-1.8cm] (u31) at (u3) {};
\node[vertex, fill=white, yshift=-.8cm] (u32) at (u31){};
\node[tinyvertex, yshift=-1cm] (u33) at (u32){};
\node[yshift=-.4cm,  font=\tiny,rotate=90] at (u33) {$\cdots$};
\node[tinyvertex, yshift=-.8cm] (u34) at (u33){};

\node[vertex, fill=white, label={ below:$x$}, yshift=-1.8cm] (u41) at (u4) {};
\node[tinyvertex, yshift=-1.8cm] (u42) at (u41){};
\node[yshift=-.4cm,  font=\tiny,rotate=90] at (u42) {$\cdots$};
\node[tinyvertex, yshift=-.8cm] (u43) at (u42){};

\node[vertex, fill=white, yshift=-1.8cm] (u51) at (u5) {};
\node[tinyvertex, yshift=-1.8cm] (u52) at (u51){};
\node[yshift=-.4cm,  font=\tiny,rotate=90] at (u52) {$\cdots$};
\node[tinyvertex, yshift=-.8cm] (u53) at (u52){};
\node [yshift=-1.5cm] at (u34) {$v_2 \to v_3$};
\node (h2) at (u53)[xshift=1cm, yshift=-1.3cm]{};
\draw[dashed] (h1) -- (h2);

\end{scope}

\path[-]
(u11) edge[<-, bend left] (u31)
	  edge[->] (u22)
	  edge[->,bend left] (u41)

(u21) [<-] edge (u32);

\begin{pgfonlayer}{background}
\node[ellipse, draw=black, fit=(u21)(u22)(u23)(u24), xshift=-2cm]{};
\node[ellipse, draw=black, fit=(u21)(u22)(u23)(u24)]{};
\node[ellipse, draw=black, fit=(u31)(u32)(u34)]{};
\node[ellipse, draw=black, fit=(u21)(u22)(u23)(u24), xshift=4cm]{};
\node[ellipse, draw=black, fit=(u21)(u22)(u23)(u24), xshift=6cm]{};
\end{pgfonlayer}

\end{scope}

\begin{scope}[yshift=-8cm,xshift=10cm]
\node[vertex, label={below:$v_1$}] (u1){};
\node[vertex, label={below:$v_2$}, xshift=2cm] (u2) at (u1) {};
\node[vertex, label={below:$v_3$},  xshift=2cm] (u3) at (u2) {};
\node[vertex,  xshift=2cm]  (u4) at (u3) {};
\node[xshift=1cm] at (u4) {$\cdots$};
\node[vertex, xshift=2cm] (u5) at (u4) {};

\path[-]
		(u1)
		edge [->] (u2)
		edge [<-,bend left] (u3)
		edge [->, bend left] (u4)
		(u2)
		edge[<-] (u3);
%

\begin{scope}[bend angle=10]
\node[vertex, label={ below:$x_1$}, yshift=-1.8cm] (u11) at (u1) {};
\node[tinyvertex, yshift=-1.8cm] (u12) at (u11){};
\node[yshift=-.4cm,  font=\tiny,rotate=90] at (u12) {$\cdots$};
\node[tinyvertex, yshift=-.8cm] (u13) at (u12){};
\node [xshift=-.5cm] at (u11) {\Lightning};

\node[vertex, label={ below:$x_2$}, yshift=-1.8cm] (u21) at (u2) {};
\node[vertex, fill=white, label={ below:$x_2'$}, yshift=-.8cm] (u22) at (u21){};
\node[tinyvertex, yshift=-1cm] (u23) at (u22){};
\node[yshift=-.4cm,  font=\tiny,rotate=90] at (u23) {$\cdots$};
\node[tinyvertex, yshift=-.8cm] (u24) at (u23){};

\node[vertex, label={ below:$x_3$}, yshift=-1.8cm] (u31) at (u3) {};
\node[vertex, fill=white, yshift=-.8cm] (u32) at (u31){};
\node[tinyvertex, yshift=-1cm] (u33) at (u32){};
\node[yshift=-.4cm,  font=\tiny,rotate=90] at (u33) {$\cdots$};
\node[tinyvertex, yshift=-.8cm] (u34) at (u33){};

\node[vertex, fill=white, label={ below:$x$}, yshift=-1.8cm] (u41) at (u4) {};
\node[tinyvertex, yshift=-1.8cm] (u42) at (u41){};
\node[yshift=-.4cm,  font=\tiny,rotate=90] at (u42) {$\cdots$};
\node[tinyvertex, yshift=-.8cm] (u43) at (u42){};

\node[vertex, fill=white, yshift=-1.8cm] (u51) at (u5) {};
\node[tinyvertex, yshift=-1.8cm] (u52) at (u51){};
\node[yshift=-.4cm,  font=\tiny,rotate=90] at (u52) {$\cdots$};
\node[tinyvertex, yshift=-.8cm] (u53) at (u52){};
\node [yshift=-1.5cm] at (u34) {$v_1 \to v_2$};

\end{scope}

\path[-]
(u11) edge[<-, bend left] (u31)
	  edge[->] (u22)
	  edge[->,bend left] (u41)

(u21) [<-] edge (u32);

\begin{pgfonlayer}{background}
\node[ellipse, draw=black, fit=(u21)(u22)(u23)(u24), xshift=-2cm]{};
\node[ellipse, draw=black, fit=(u21)(u22)(u23)(u24)]{};
\node[ellipse, draw=black, fit=(u31)(u32)(u34)]{};
\node[ellipse, draw=black, fit=(u21)(u22)(u23)(u24), xshift=4cm]{};
\node[ellipse, draw=black, fit=(u21)(u22)(u23)(u24), xshift=6cm]{};
\end{pgfonlayer}

\end{scope}
\end{tikzpicture}

%% file: claim7-fig1.tex
\begin{tikzpicture}[>={[scale=1.1]Stealth},
tinyvertex/.style={circle,minimum size=1mm,very thick, draw=black, fill=black, inner sep=0mm}]

\node[vertex, label={below:$w$}] (w){};
\node[vertex, xshift=-3cm] (u') at (w) {};
\node[vertex, label={below:$v$}, xshift=2cm] (v) at (w) {};
\node[vertex, label={below:$w'$},  xshift=2cm] (w') at (v) {};
\node[xshift=1.5cm] at (w') {$\cdots$};
\node[vertex,  xshift=3cm, label=below:$u$]  (u) at (w') {};
\node[vertex, xshift=2cm] (x) at (u) {};
\node[xshift=1cm] at (x) {$\cdots$};
\node[vertex, xshift=2cm] (x') at (x) {}; 
\node [xshift=-1cm](wl) at (w) {};
\node[xshift=-1.5cm] at (w) {$\cdots$};
\node [xshift=1cm] (wr) at (w'){};
\node[yshift=.4cm, xshift=.4cm] at (w) {$C_1$};
\node[yshift=.4cm, xshift=.4cm] at (w') {$C_2$};

\path[-, bend angle=30]
		(v)
		edge [->] (w')
		edge [->] (w)
		edge [<-, bend left] (u.west)
		(w) edge [->] (wl)
		(w') edge[->] (wr)
		(u') edge [->, bend angle=40, bend left] node [midway, label={above:$C$}] {} (u.130);
%

\begin{scope}[bend angle=10]
\node[vertex, label={ below:$v_1$}, yshift=-1.8cm] (v1) at (v) {};
\node[vertex, label={ below:$v_2$}, yshift=-.8cm] (v2) at (v1){};
\node[vertex, label={ below:$v_3$}, yshift=-.8cm] (v3) at (v2){};
\node[tinyvertex, yshift=-.8cm] (v4) at (v3){};
\node[yshift=-.4cm,  font=\tiny,rotate=90] at (v4) {$\cdots$};
\node[tinyvertex, yshift=-.8cm] (v5) at (v4){};

\node[vertex, fill=white, label={ below:$w_1$}, yshift=-1.8cm] (w1) at (w) {};
\node[vertex, label={ below:$w_2$}, yshift=-.8cm] (w2) at (w1){};
\node[yshift=-.8cm] (w3) at (w2){};
\node[tinyvertex, yshift=-.8cm] (w4) at (w3){};
\node[yshift=-.4cm,  font=\tiny,rotate=90] at (w4) {$\cdots$};
\node[tinyvertex, yshift=-.8cm] (w5) at (w4){};

\node[vertex, fill=white, label={[yshift=5pt] below:$w_1'$}, yshift=-1.8cm] (w1') at (w') {};
\node[vertex, label={[yshift=5pt] below:$w_2'$}, yshift=-.8cm] (w2') at (w1'){};
\node[yshift=-.8cm] (w3') at (w2'){};
\node[tinyvertex, yshift=-.8cm] (w4') at (w3'){};
\node[yshift=-.4cm,  font=\tiny,rotate=90] at (w4') {$\cdots$};
\node[tinyvertex, yshift=-.8cm] (w5') at (w4'){};

\node[vertex, fill=white, label={[xshift=3pt]below:$u_1$}, yshift=-1.8cm] (u1) at (u) {};
\node[tinyvertex, yshift=-2.4cm] (u4) at (u1){};
\node[yshift=-.4cm,  font=\tiny,rotate=90] at (u4) {$\cdots$};
\node[tinyvertex, yshift=-.8cm] (u5) at (u4){};

\node[vertex, fill=white,  yshift=-1.8cm] (u1') at (u') {};
\node[tinyvertex, yshift=-2.4cm] (u4') at (u1'){};
\node[yshift=-.4cm,  font=\tiny,rotate=90] at (u4') {$\cdots$};
\node[tinyvertex, yshift=-.8cm] (u5') at (u4'){};

\node[vertex, fill=white,  yshift=-1.8cm] (x1) at (x) {};
\node[tinyvertex, yshift=-2.4cm] (x4) at (x1){};
\node[yshift=-.4cm,  font=\tiny,rotate=90] at (x4) {$\cdots$};
\node[tinyvertex, yshift=-.8cm] (x5) at (x4){};

\node[vertex, fill=white,  yshift=-1.8cm] (x1') at (x') {};
\node[tinyvertex, yshift=-2.4cm] (x4') at (x1'){};
\node[yshift=-.4cm,  font=\tiny,rotate=90] at (x4') {$\cdots$};
\node[tinyvertex, yshift=-.8cm] (x5') at (x4'){};
\end{scope}	

\begin{pgfonlayer}{background}
\node[ellipse, draw=black, fit=(u1')(u5')]{};
\node[ellipse, draw=black, fit=(w1)(w5)]{};
\node[ellipse, draw=black, fit=(v1)(v5)]{};
\node[ellipse, draw=black, fit=(w1')(w5')]{};
\node[ellipse, draw=black, fit=(u1)(u5)]{};
\node[ellipse, draw=black, fit=(x1)(x5)]{};
\node[ellipse, draw=black, fit=(x1')(x5')]{};
\end{pgfonlayer}

\node[vertex, fill=white, yshift=-1.5cm, xshift=.5cm] (wv) at (w5) {};
\node [xshift=1.6cm] at (wv) {-- vertices of $T$};
\path[-]
(v1) edge[->] (w1)
	 edge[->](w2')
	
(v2) edge[->] (w2)
		edge[->] (w1')
(u1) edge[->,bend angle=55, bend left] (v3);
\end{tikzpicture}

%% file: claim7-fig2.tex
\begin{tikzpicture}[>={[scale=1.1]Stealth},
tinyvertex/.style={circle,minimum size=1mm,very thick, draw=black, fill=black, inner sep=0mm}]

\node[vertex, label={below:$w$}] (w){};
\node[vertex, xshift=-3cm] (u') at (w) {};
\node[vertex, label={below:$v$}, xshift=2cm] (v) at (w) {};
\node[vertex, label={below:$w'$},  xshift=2cm] (w') at (v) {};
\node[xshift=1.5cm] at (w') {$\cdots$};
\node[vertex,  xshift=3cm, label=below:$u$]  (u) at (w') {};
\node[vertex, xshift=2cm] (x) at (u) {};
\node[xshift=1cm] at (x) {$\cdots$};
\node[vertex, xshift=2cm] (x') at (x) {}; 
\node [xshift=-1cm](wl) at (w) {};
\node[xshift=-1.5cm] at (w) {$\cdots$};
\node [xshift=1cm] (wr) at (w'){};
\node[yshift=.4cm, xshift=.4cm] at (w) {$C_1$};
\node[yshift=.4cm, xshift=.4cm] at (w') {$C_2$};

\path[-, bend angle=30]
		(v)
		edge [->] (w')
		edge [->] (w)
		edge [<-, bend left] (u.west)
		(w) edge [->] (wl)
		(w') edge[->] (wr)
		(u') edge [->, bend angle=40, bend left] node [midway, label={above:$C$}] {} (u.130);
%

\begin{scope}[bend angle=10]
\node[vertex, label={ below:$v_1$}, yshift=-1.8cm] (v1) at (v) {};
\node[vertex, label={ below:$v_2$}, yshift=-.8cm] (v2) at (v1){};
\node[vertex, label={ below:$v_3$}, yshift=-.8cm] (v3) at (v2){};
\node[tinyvertex, yshift=-.8cm] (v4) at (v3){};
\node[yshift=-.4cm,  font=\tiny,rotate=90] at (v4) {$\cdots$};
\node[tinyvertex, yshift=-.8cm] (v5) at (v4){};

\node[vertex, fill=white, label={ below:$w_1$}, yshift=-1.8cm] (w1) at (w) {};
\node[vertex, label={ below:$w_2$}, yshift=-.8cm] (w2) at (w1){};
\node[vertex, label={ below:$w_3$}, yshift=-.8cm] (w3) at (w2){};
\node[tinyvertex, yshift=-.8cm] (w4) at (w3){};
\node[yshift=-.4cm,  font=\tiny,rotate=90] at (w4) {$\cdots$};
\node[tinyvertex, yshift=-.8cm] (w5) at (w4){};

\node[vertex, fill=white, label={[yshift=5pt] below:$w_1'$}, yshift=-1.8cm] (w1') at (w') {};
\node[vertex, label={[yshift=5pt]below:$w_2'$}, yshift=-.8cm] (w2') at (w1'){};
\node[vertex, label={ [yshift=5pt]below:$w_3'$}, yshift=-.8cm] (w3') at (w2'){};
\node[tinyvertex, yshift=-.8cm] (w4') at (w3'){};
\node[yshift=-.4cm,  font=\tiny,rotate=90] at (w4') {$\cdots$};
\node[tinyvertex, yshift=-.8cm] (w5') at (w4'){};

\node[vertex, fill=white, label={[xshift=3pt]below:$u_1$}, yshift=-1.8cm] (u1) at (u) {};
\node[tinyvertex, yshift=-2.4cm] (u4) at (u1){};
\node[yshift=-.4cm,  font=\tiny,rotate=90] at (u4) {$\cdots$};
\node[tinyvertex, yshift=-.8cm] (u5) at (u4){};

\node[vertex, fill=white,  yshift=-1.8cm] (u1') at (u') {};
\node[tinyvertex, yshift=-2.4cm] (u4') at (u1'){};
\node[yshift=-.4cm,  font=\tiny,rotate=90] at (u4') {$\cdots$};
\node[tinyvertex, yshift=-.8cm] (u5') at (u4'){};

\node[vertex, fill=white,  yshift=-1.8cm] (x1) at (x) {};
\node[tinyvertex, yshift=-2.4cm] (x4) at (x1){};
\node[yshift=-.4cm,  font=\tiny,rotate=90] at (x4) {$\cdots$};
\node[tinyvertex, yshift=-.8cm] (x5) at (x4){};

\node[vertex, fill=white,  yshift=-1.8cm] (x1') at (x') {};
\node[tinyvertex, yshift=-2.4cm] (x4') at (x1'){};
\node[yshift=-.4cm,  font=\tiny,rotate=90] at (x4') {$\cdots$};
\node[tinyvertex, yshift=-.8cm] (x5') at (x4'){};
\end{scope}	

\begin{pgfonlayer}{background}
\node[ellipse, draw=black, fit=(u1')(u5')]{};
\node[ellipse, draw=black, fit=(w1)(w5)]{};
\node[ellipse, draw=black, fit=(v1)(v5)]{};
\node[ellipse, draw=black, fit=(w1')(w5')]{};
\node[ellipse, draw=black, fit=(u1)(u5)]{};
\node[ellipse, draw=black, fit=(x1)(x5)]{};
\node[ellipse, draw=black, fit=(x1')(x5')]{};
\end{pgfonlayer}

\node[vertex, fill=white, yshift=-1.5cm, xshift=.5cm] (wv) at (w5) {};
\node [xshift=1.6cm] at (wv) {-- vertices of $T$};
\path[-]
(v1) edge[->] (w1)
	    edge[->](w2')
(v2) edge[->] (w2)
		edge[->] (w1')
(v3) edge[->] (w3)
	 edge[->] (w3')
(u1) edge[->,bend angle=55, bend left] (v3);
\end{tikzpicture}